\documentclass[11pt,a4paper]{article}
\usepackage[utf8]{inputenc}
\usepackage{amsmath}
\usepackage{amsfonts}
\usepackage{amssymb}
\usepackage{amsthm}
\usepackage{makeidx}
\usepackage{graphicx}
\usepackage{authblk}
\usepackage{subfig}
\usepackage{subeqnarray}
\usepackage[T1]{fontenc}
\usepackage[top=2cm, bottom=2cm, left=2cm, right=2cm]{geometry}
\usepackage{marginnote}

\usepackage{imakeidx}
\usepackage{hyperref}
\usepackage{color}
\makeindex[intoc]

\title{The p-Laplacian equation in thin domains: The unfolding approach}
\date{2018}
\newtheorem{theorem}{Theorem}[section]
\newtheorem{corollary}{Corollary}[theorem]

\newtheorem{definition}[theorem]{Definition}
\newtheorem{proposition}[theorem]{Proposition}
\newtheorem{remark}{Remark}[section]

\author[1]{Jos\'e M. Arrieta\thanks{e-mail: arrieta@mat.ucm.es.}}
\author[2]{Jean Carlos Nakasato\thanks{e-mail: nakasato@ime.usp.br.}}
\author[2]{Marcone Corr\^ea Pereira\thanks{e-mail: marcone@ime.usp.br.}}
\affil[1]{Depto. de An\'alisis Matem\'atico y Matem\'atica Aplicada,  Universidad Complutense de Madrid, 28040 Madrid, Spain}
\affil[2]{Depto. Matem\'atica Aplicada, IME, Universidade de S\~ao Paulo,
	Rua do Mat\~ao 1010, S\~ao Paulo - SP, Brazil}

\begin{document}
	\maketitle
		\begin{abstract}
		In this work we apply the unfolding operator method to analyze the asymptotic behavior of the solutions of the $p$-Laplacian equation with Neumann boundary condition set in a bounded thin domain of the type 
		$R^\varepsilon=\left\lbrace(x,y)\in\mathbb{R}^2:x\in(0,1)\mbox{ and }0<y<\varepsilon g\left({x}/{\varepsilon^\alpha}\right)\right\rbrace$ where $g$ is a positive periodic function.
		%We take a $L$-periodic function $g:\mathbb{R} \mapsto \mathbb{R}$ in $L^\infty(\mathbb{R})$. The thin domain situation is established passing to the limit in the positive parameter $\varepsilon$ with $\varepsilon \to 0$. 
		We study the three cases $0<\alpha<1$, $\alpha=1$ and $\alpha>1$ representing respectively weak, resonant and high osillations at the top boundary. In the three cases we deduce the homogenized limit and obtain correctors. 	\end{abstract}

	\noindent \emph{Keywords:} $p$-Laplacian, Neumann boundary condition, Thin domains, Oscillatory boundary, Homogenization. \\
	\noindent 2010 \emph{Mathematics Subject Classification.} 35B25, 35B40, 35J92.

	%\footnotetext[1]{This work was supported by CNPq}
	
	\section{Introduction}

Let $R^\varepsilon \subset \mathbb{R}^2$ be the following family of thin domains  
	\begin{equation} \label{TDs}
	R^\varepsilon=\left\lbrace(x,y)\in\mathbb{R}^2:x\in(0,1)\mbox{ and }0<y<\varepsilon g\left(\frac{x}{\varepsilon^\alpha}\right)\right\rbrace, \qquad \varepsilon>0,
	\end{equation}
	where $\alpha>0$ is a fixed parameter, 
	$g:\mathbb{R}\rightarrow\mathbb{R}$ is a strictly positive function, periodic of period $L$, lower semicontinuous satisfying  
	\begin{equation*}  \label{g0}
	0<g_0 \leq g(x)\leq g_1, \quad \forall x \in (0,L),
	\end{equation*}
	with $g_0 = \min_{x \in \mathbb{R}} g(x)$ and $g_1 = \sup_{x \in \mathbb{R}} g(x)$. 
	
	In this work, we are interested in analyzing the asymptotic behavior of the family of solutions of the nonlinear elliptic equation  
	\begin{equation}\label{problem01}
	\left\lbrace \begin{array}{ll}
	-\Delta_p u_\varepsilon +|u_\varepsilon|^{p-2}u_\varepsilon=f^\varepsilon\mbox{ in }R^\varepsilon \\
	|\nabla u_\varepsilon|^{p-2}\nabla u_\varepsilon \eta_\varepsilon =0\mbox{ on }\partial R^\varepsilon
	\end{array}\right.
	\end{equation}
	where $\eta_\varepsilon$ is the unit outward normal vector to the boundary $\partial R^\varepsilon$, $1<p<\infty$
	 and 
	$$
	\Delta_p\cdot = \mbox{div }\left(|\nabla \cdot|^{p-2}\nabla\cdot\right)
	$$
	denotes the $p$-Laplacian differential operator. We also assume $f^\varepsilon\in L^{p'}(R^\varepsilon)$ where $p'$ is the conjugate exponent of $p$, that is $1/p'+1/p=1$.

	It is known that the variational formulation of \eqref{problem01} is given by
	\begin{equation}\label{variationalproblem01}
	\int_{R^\varepsilon} \left\{ |\nabla u_\varepsilon|^{p-2}\nabla u_\varepsilon\nabla \varphi +|u_\varepsilon|^{p-2}u_\varepsilon\varphi \right\}dxdy=\int_{R^\varepsilon} f^\varepsilon\varphi \, dxdy, \qquad \varphi \in W^{1,p}(R^\varepsilon).
	\end{equation}
	Furthermore,  for each fixed $\varepsilon>0$ the existence and uniqueness of solutions is guaranteed by Minty-Browder's Theorem.  
	Hence, we are interested here in analyzing the asymptotic behavior of the solutions $u_\varepsilon$ as $\varepsilon \to 0$, that is, as the domain $R^\varepsilon$ becomes thinner and thinner although with a high oscillating boundary at the top.
	
	Indeed, since the set $R^\varepsilon \subset (0,1) \times (0,\varepsilon \, g_1 )$ for all $\varepsilon>0$, we have the parameter $\varepsilon$ models the thin domain situation. 
	Moreover, we see that $R^\varepsilon$ has tickness of order $\varepsilon$, and then, it is expected that for $\varepsilon \approx 0$ the sequence of solutions $u_\varepsilon$ will converge to a function of just one single variable $x \in (0,1)$ and that this function will satisfy an equation of the same type as \eqref{problem01} but in one dimension.
	
	On the other side, the parameter $\alpha$ measures the intensity of the oscillations of the top boundary and, as we will see, the homogenized limit equation will depend tightly on this positive number. 
We will deal with three distinct cases: weak oscillatory case ($0<\alpha<1$), the resonant or critical case ($\alpha = 1$), and the high oscillatory one ($\alpha>1$) See Figure \ref{fig1} below where these three cases are illustrated. We will obtain different limit problems according to this three cases.

\begin{figure}[!htb]
\centering
\subfloat[Weak oscillation: $\alpha<1$.]{
\scalebox{0.43}{\includegraphics{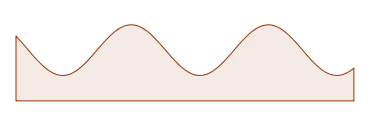}}
\label{figdroopy}
}
\quad %espaco separador
\subfloat[Resonant case: $\alpha=1$.]{
\scalebox{0.43}{\includegraphics{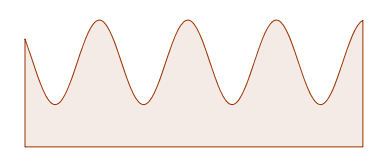}}
\label{figsnoop}
}
\quad
\subfloat[High oscillation: $\alpha>1$.]{
\scalebox{0.5}{\includegraphics{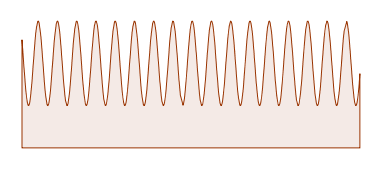}}
\label{figdroopy1}
}
\caption{Examples of oscillatory behavior at the boundary $\partial R^\varepsilon$.}
\label{fig1}
\end{figure}

	Here we will combine techniques such as unfolding operator methods for thin domains, which were developed in \cite{AM,AM2}, as well as, those ones presented in \cite{MasoA,donato1990} in order to analyze monotone operators in perforated domains. Furthermore, we will also obtain corrector results for each case considered here.

	As we will see in this work, the homogenized limit problem is given by the following one-dimensional $p$-Laplacian equation with constant coefficient $q$:
		\begin{equation} \label{limeq}
	\left\lbrace \begin{gathered}
	- q \left(|u'|^{p-2}u'\right)' +|u|^{p-2}u=\bar{f} \quad \mbox{ in }(0,1)\\
	u'(0)=u'(1)=0.
	\end{gathered}\right. .
	\end{equation}

	Indeed, the coefficient $q$ has different expresion for the three different cases of $\alpha$.  As a matter of fact,  for $\alpha = 1$ in \eqref{TDs}, we show that the homogenized coefficient $q$ is a positive constant and it is given by 
	\begin{eqnarray}  \label{coeq}
	q&=&\begin{array}{ll}
		\dfrac{1}{|Y^*|}\displaystyle\int_{Y^*}|\nabla v|^{p-2} \, \partial_{y_1} v \, dy_1dy_2  
	\end{array}
	\end{eqnarray}
	where $Y^*$ is the representative cell of the oscillating domain $R^\varepsilon$
	\begin{equation}\label{basic-cell} 
	Y^*=\left\lbrace \right(y_1,y_2) \in \mathbb{R}^2 : 0<y_1<L\mbox{ and }0<y_2<g(y_1)\rbrace
	\end{equation}
 %with Lebesgue measure $|Y^*|$. 
 %
The function $v$ appearing in \eqref{coeq} is an auxiliar function, which is the unique solution of the following problem  
	\begin{equation}\label{auxiliarproblemi}
	\begin{gathered}
	\displaystyle\int_{Y^*}\left|\nabla v\right|^{p-2}\nabla v\nabla \varphi \, dy_1dy_2=0 \quad \forall\varphi\in W^{1,p}_{\#}(Y^*),\\
	(v-y_1)\in  W^{1,p}_{\#}(Y^*) \qquad \textrm{ with } \quad \left\langle (v-y_1) \right\rangle_{Y^*} = 0,
	\end{gathered}
	\end{equation}
	where  
	$$
	W_{\#}^{1,p}(Y^*) = \{ \varphi\in W^{1,p}(Y^*) \, : \,  \varphi |_{\partial_{left} Y^*} = \varphi|_{\partial_{right} Y^*} \}
	$$
	is the space of periodic functions on the horizontal variable $x$, and 	$\left\langle \varphi\right\rangle_{\mathcal{O}}$ denotes the average of the function $\varphi \in L^1_{loc}(\mathbb{R}^N)$ on the open set $\mathcal{O} \subset \mathbb{R}^N$.

		It is worth noting that problem \eqref{auxiliarproblemi} is well posed due again to Minty-Browder's Theorem. This implies that $q$ is well defined and is a positive constant  (see \eqref{q>0}).
		Moreover, we have that the forcing term $\bar f$ of the limit equation \eqref{limeq} is obtained as the limit of the unfolding operator acting on functions $f^\varepsilon$ (see for instance \eqref{fhat01} and \eqref{falpha1} below).

	\par\medskip 
	
	For the case $\alpha<1$, 
	we obtain that the homogenized coefficient depends just on the function $g$, which describes the profile of the oscillatory boundary and on the number $p \in (1,\infty)$, which establishes the order of the $p$-Laplacian operator. It gets the following form   
	\begin{equation*}  
	q=\dfrac{1}{\langle g\rangle_{(0,L)}\left\langle 1/g^{p'-1}\right\rangle_{(0,L)}^{p-1}}.
	\end{equation*}
	We still mention that the forcing term $\bar f$ is also given by \eqref{fhat01} and \eqref{falpha1} since it is computed in the same way that in the previous case $\alpha=1$.

	\par\medskip

	Finally, for the case $\alpha>1$, we first note that forcing term $\bar f$ gets a different expression. Since it is computed in a different way, not anymore as a consequence of the unfolding operator, it takes the form \eqref{falpha>1}. The homogenized coefficient $q$ of the limit equation \eqref{limeq} now assumes the form
	\begin{equation*}
	q=\frac{g_0}{\langle g\rangle_{(0,L)}}.
	\end{equation*}
	It does not depend explicitly on $p$, but  on $g_0$, the minimum value of the $L$-periodic function $g$,  which is strictly positive. 
	For this case, it is easy to see that $q<1$ if $g$ is not constant. In fact, 
	$$
	\frac{g_0}{\langle g\rangle_{(0,L)}} = \frac{Lg_0}{|Y^*|} < \frac{|Y^*|}{|Y^*|} = 1.
	$$
	Somehow, we can say that the high oscillatory behavior tends to affect the system in such way that its diffusion becomes smaller. 
	Notice that $q$ also has a lower bound in the class of functions $g$ considered here. It satisfies $q \geq g_0/g_1$ where $g_1$ is the maximum value of $g$ in $[0,L]$.

	Complete statements on the homogenized limit problems and the corresponding convergence of solutions are stated in Theorem \ref{resultressonant} for $\alpha=1$, Theorem \ref{weakresult} for $0<\alpha<1$,  and Theorem \ref{strongresult} for $\alpha>1$.
	Strong convergence in Sobolev spaces like $W^{1,p}$ are also obtained using the corrector approach discussed for instance in \cite{AM2,MasoA}. 
	We show the existence of a family of functions $W_\varepsilon$ such that  
	$$
	\varepsilon^{-1} ||\nabla u_\varepsilon-W_\varepsilon  ||^p_{L^p(R^\varepsilon)}\rightarrow 0, \quad \textrm{ as } \varepsilon \to 0.
	$$
	Such results are precisely stated in Corollary \ref{cor1045}, \ref{correctoralfa<1} and \ref{correcal>1}, respectively for each case: $\alpha=1$, $0<\alpha<1$ and $\alpha>1$.

	Now, let us notice that there are several works in the literature dealing with issues related to the effect of thickness and roughness on the feature of the solutions of partial differential equations. 	
	Indeed, thin structures with oscillating boundaries appear in many fields of science: fluid dynamics (lubrication), solid mechanics (thin rods, plates or shells) or even physiology (blood circulation). Therefore, analyzing the asymptotic behavior of different models on thin structures and understanding how the geometry and the roughness affects the limit problem is a very relevant issues in applied science. 
	Here, we just mention some works in these directions \cite{MGrau,BFN,BPazanin,GH0,GH,MPov,PazaninGrau,MZAMP}.

	 Furthermore, we point out that the particular case taking $p=2$ in equation \eqref{problem01}, which represents the Laplace differential operator, has been originally discussed in some previous works using different techniques and methods.   Indeed, in \cite{Villanueva2016} the author among other things,  treats the case of $0<\alpha<1$ (even with a doubly oscillatory boundary) via change of variables and rescaling the thin domain as in the classical work \cite{HaleRaugel}. The resonant case, $\alpha=1$, has been studied in \cite{arrieta2011,AP10,MPov} where techniques from homogenization theory have been used.  
	
	The case with fast oscillatory boundary ($\alpha > 1$) was obtained in \cite{AM01} by decomposing the domain in two parts separating the oscillatory boundary. There, the authors also consider more general and complicated geometries which are not given as the graph of certain smooth functions. See also \cite{AM2, Villanueva2016}.  
	
	In \cite{AM,AM2}, the authors introduce the unfolding method in thin domains tackling these three cases for the Laplace operator with Neumann boundary condition in a unified way. Also, the regularity requirement on the function $g$ is very mild.
	
	\par\medskip
	
	Now, concerning with the $p$-Laplacian, we have recently applied techniques from \cite{arrieta2011,MasoA,donato1990} to obtain the limit equation and corrector results in smooth thin domains for the resonant case in \cite{nakasato} improving previous works such as \cite{MRi2}. 
	
	The main goal of this paper is to improve the previous mentioned works considering the p-Laplacian equation. As a matter of fact, combining techniques from \cite{MasoA,donato1990} and \cite{AM,AM2}, we will be able to deal with equation \eqref{problem01} on non-smooth oscillating thin domains for any $1<p<\infty$ and any order of oscillation $\alpha>0$. 
	
	Notice that this is not a easy task since we are studying here a quasilinear differential equation, which can be singular, as it is in the case $1<p<2$, or degenerated, if $2<p<\infty$. Moreover, the problem is posed in non-smooth thin domains like comb-like thin domains where standard extension operators do not apply (see figure \ref{fig21}). Besides, it is worth observing that the unfolding method also allows us to obtain some new strong convergence results for the solutions by corrector approach. 
	
	\begin{figure}[!htb]
		\centering{
			\scalebox{0.35}{\includegraphics{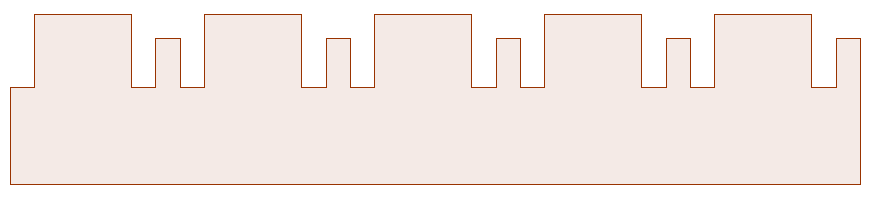}}
		}
		\caption{Comb-like thin domains.}
		\label{fig21}
	\end{figure}
	
	The paper is organized as follows: in Section \ref{USection}, we state some notations and basic results. In Section \ref{seal1}, we consider the resonant case $\alpha=1$ obtaining the homogenized equation via a somehow classical auxiliar problem given by homogenization theory. In Section \ref{seal<1}, the weak oscillation case is studied, and in Section \ref{seal>1}, we finally consider the case of thin domains with very highly oscillatory boundaries.

	\section{Notations and Basic Facts}  \label{USection}
	
	To study the convergence of the solutions of \eqref{variationalproblem01}, we clarify some notation and recall some results concerning monotone operators and the method of unfolding operator. We will need these results for our analysis.
	
	We consider two-dimensional thin domains defined by \eqref{TDs}. Observe that this domains have an oscillatory behavior at its top boundary. The parameters $\varepsilon$ and $\alpha$ are positive and the function $g$ satisfies the following hypothesis

\par\medskip 	

	($\mathbf{H_g}$) {\sl $g : \mathbb{R}\rightarrow \mathbb{R}$ is a strictly positive, bounded, lowersemicontinuous, $L$-periodic function.  Moreover, we define $$g_0 = \min_{x \in \mathbb{R}} g(x) \quad  \textrm{ and } \quad g_1 = \max_{x \in \mathbb{R}} g(x)$$ 
so that $0<g_0\leq g(x)\leq g_1$ for all $x\in\mathbb{R}$}

%that is, there exist positive constants $g_0$ and $g_1$ such that $$0 < g_0\leq g(x)\leq g_1, \quad \textrm{ for all } x\in \mathbb{R},$$
%	with 
%	$$g_0 = \min_{x \in \mathbb{R}} g(x) \quad  \textrm{ and } \quad g_1 = \max_{x \in \mathbb{R}} g(x).$$ 
	
	\par\medskip Recall that lower semicontinuous means that $g(x_0)\leq \displaystyle\liminf_{x\rightarrow x_0}g(x)$, $\forall x_0\in\mathbb{R}$.

Recall that  $Y^*$, given by \eqref{basic-cell}
%	\begin{equation*} 
%	Y^*=\left\lbrace \right(y_1,y_2) \in \mathbb{R}^2 : 0<y_1<L\mbox{ and }0<y_2<g(y_1)\rbrace
%	\end{equation*}
	is the basic cell of the thin domain $R^\varepsilon$ and 
	$$
	\left\langle \varphi\right\rangle_{\mathcal{O}} := \frac{1}{|\mathcal{O}|} \int_{\mathcal{O}} \varphi(x) \, dx
	$$
	is the average of $\varphi \in L^1_{loc}(\mathbb{R}^2)$ on an open bounded set $\mathcal{O} \subset \mathbb{R}^2$. 
	
	We will also need to consider the following functional spaces which are defined by periodic functions in the variable $y_1 \in (0,L)$. Namely 
	$$
	\begin{gathered}
	L^p_\#(Y^*) = \{ \varphi \in L^p(Y^*) \, : \, \varphi(y_1,y_2) \textrm{ is $L$-periodic in $y_1$ } \}, \\
	 L^p_\#\left((0,1)\times Y^*\right) =  \{ \varphi \in L^p((0,1) \times Y^*) \, : \, \varphi(x, y_1,y_2) \textrm{ is $L$-periodic in $y_1$ } \}, \\
	W_{\#}^{1,p}(Y^*) = \{ \varphi\in W^{1,p}(Y^*) \, : \,  \varphi |_{\partial_{left} Y^*} = \varphi|_{\partial_{right} Y^*}\}.
	\end{gathered}
	$$
	
	If we denote by $[a]_L$ the unique integer number such that $a=[a]_L L+\{a\}_L$ where $\{a\}_L\in [0,L)$, then for each $\varepsilon>0$ and any $x\in\mathbb{R}$, we have 
	\begin{equation*}
	x=\varepsilon^\alpha \left[\frac{x}{\varepsilon^\alpha}\right]_LL+\varepsilon^\alpha\left\{\frac{x}{\varepsilon^\alpha}\right\}_L\mbox{ where }\left\{\frac{x}{\varepsilon^\alpha}\right\}_L\in [0,L).
	\end{equation*}
	
	Let us also denote 
	\begin{equation*}
	I_\varepsilon=\mbox{ Int }\left(\bigcup_{k=0}^{N_\varepsilon}\left[kL\varepsilon^\alpha,(k+1)L\varepsilon^\alpha\right]\right), 
	\end{equation*}
	where $N_\varepsilon$ is the largest integer such that $\varepsilon^\alpha L(N_\varepsilon+1)\leq 1$. 
	We still set 
	$$
	\begin{gathered}
	\Lambda_\varepsilon=(0,1)\backslash I_\varepsilon=[\varepsilon^\alpha L(N_\varepsilon+1),1), \\
	R^\varepsilon_0=\left\lbrace(x,y)\in \mathbb{R}^2:x\in I_\varepsilon, 0<y<\varepsilon^\alpha g\left(\frac{x}{\varepsilon^\alpha}\right)\right\rbrace, \\
	R^\varepsilon_1=\left\lbrace(x,y)\in \mathbb{R}^2:x\in \Lambda_\varepsilon, 0<y<\varepsilon^\alpha g\left(\frac{x}{\varepsilon^\alpha}\right)\right\rbrace.
	\end{gathered}
	$$
	
	Observe that we  have $\Lambda_\varepsilon=\emptyset\,\,$  if $\varepsilon^\alpha L(N_\varepsilon+1)=1$. In this case $R_0^\varepsilon=R^\varepsilon$ and $R_1^\varepsilon=\emptyset$.

	The following well known inequalities will be needed throughout the paper (see \cite{lindqvist}).
	
	\begin{proposition}\label{proposicaoplaplaciano}
		Let $x,y\in\mathbb{R}^n$.
		\begin{itemize}
			\item If $p\geq 2$, then
			\begin{equation*}\label{eq1}
			\langle|x|^{p-2}x-|y|^{p-2}y,x-y\rangle \geq c_p|x-y|^p.
			\end{equation*}
			\item If $1<p<2$, then
			\begin{eqnarray*}\label{eq2}
			\langle|x|^{p-2}x-|y|^{p-2}y,x-y\rangle&\geq&c_p|x-y|^2(|x|+|y|)^{p-2}\nonumber\\ &\geq& c_p|x-y|^2(1+|x|+|y|)^{p-2}.
			\end{eqnarray*}
		\end{itemize}
	\end{proposition}

	\begin{corollary}\label{corolarioplaplaciano}
	If  for $1<p<\infty$, we denote by $a_{p}(s)=|s|^{p-2}s$ then $a_p\circ a_{p'}(s)=s$ if $\frac{1}{p}+\frac{1}{p'}=1$, that is, $a_p$ and $a_{p'}$ are inverse functions. Hence, 	
		\begin{itemize}
			\item If $1<p'< 2$ (i.e, $p\geq 2$), then
			\begin{equation*}
			\left||u|^{p'-2}u-|v|^{p'-2}v\right|\leq c|u-v|^{p'-1}.
			\end{equation*}
			\item If $p'\geq 2$ (i.e, $1<p<2$), then
			\begin{eqnarray*}
			\left||u|^{p'-2}u-|v|^{p'-2}v\right|&\leq& c|u-v|(|u|+|v|)^{p'-2}\nonumber\\
			&\leq& c|u-v|(1+|u|+|v|)^{p'-2}.
			\end{eqnarray*}
		\end{itemize}
	\end{corollary}

	Now, let us recall the definition to the unfolding operator and some of its properties. For proofs and details, see \cite{AM,AM2}.
	 
	\begin{definition}
		Let $\varphi$ be a Lebesgue-measurable function in $R^\varepsilon$. The unfolding operator $T_\varepsilon$ acting on $\varphi$ is defined as the following function in $(0,1)\times Y^*$
		\begin{eqnarray*}
		T_\varepsilon\varphi(x,y_1,y_2)=\left\lbrace\begin{array}{ll}
		\varphi\left(\varepsilon^\alpha\left[\frac{x}{\varepsilon^\alpha}\right]_LL+\varepsilon^\alpha y_1,\varepsilon y_2\right)\mbox{, for }(x,y_1,y_2)\in I_\varepsilon\times Y^*,\\
		0, \mbox{ for }(x,y_1,y_2)\in \Lambda_\varepsilon\times Y^*.
		\end{array}\right.
		\end{eqnarray*}
	\end{definition}

	\begin{proposition}\label{unfoldingproperties}
		The unfolding operator satifies the following properties:
		\begin{enumerate}
			\item $T_\varepsilon$ is linear;
			\item $T_\varepsilon(\varphi\psi)=T_\varepsilon(\varphi)T_\varepsilon(\psi)$, for all $\varphi$, $\psi$ Lebesgue mesurable in $R^\varepsilon$;
			\item $\forall\varphi\in L^p(R^\varepsilon)$, $1\leq p\leq \infty$,
			\begin{equation*}
			T_\varepsilon(\varphi)\left(x,\left\lbrace\frac{x}{\varepsilon^\alpha}\right\rbrace_L,\frac{y}{\varepsilon}\right)=\varphi(x,y),
			\end{equation*}
			for $(x,y)\in R_0^\varepsilon$.
			\item Let $\varphi$ a Lebesgue mesurable function in $Y^*$ extended periodically in the first variable. Then, $\varphi^\varepsilon(x,y)=\varphi\left(\frac{x}{\varepsilon^\alpha},\frac{y}{\varepsilon}\right)$ is mesurable in $R^\varepsilon$ and
			\begin{equation*}
			T_\varepsilon(\varphi^\varepsilon)(x,y_1,y_2)=\varphi(y_1,y_2),\forall(x,y_1,y_2)\in I_\varepsilon\times Y^*.
			\end{equation*} 
			Moreover, if $\varphi\in L^p(Y^*)$, then $\varphi^\varepsilon\in L^p(R^\varepsilon)$;
			\item Let $\varphi^\varepsilon\in L^1(R^\varepsilon)$. Then,
			\begin{eqnarray*}
				& & \frac{1}{L}\int_{(0,1)\times Y^*}T_\varepsilon(\varphi)(x,y_1,y_2)dxdy_1dy_2=\frac{1}{\varepsilon}\int_{R_0^\varepsilon}\varphi(x,y)dxdy\\
				&& \qquad \qquad = \frac{1}{\varepsilon}\int_{R^\varepsilon}\varphi(x,y)dxdy-\frac{1}{\varepsilon}\int_{R_1^\varepsilon}\varphi(x,y)dxdy;
			\end{eqnarray*}
			\item $\forall \varphi \in L^p(R^\varepsilon)$, $T_\varepsilon(\varphi)\in L^p\left((0,1)\times Y^*\right)$, $1\leq p\leq \infty$. Moreover
			\begin{equation*}
			\left|\left|T_\varepsilon(\varphi)\right|\right|_{L^p\left((0,1)\times Y^*\right)}=\left(\frac{L}{\varepsilon}\right)^{\frac{1}{p}}\left|\left|\varphi\right|\right|_{L^p(R_0^\varepsilon)}\leq \left(\frac{L}{\varepsilon}\right)^{\frac{1}{p}}\left|\left|\varphi\right|\right|_{L^p(R^\varepsilon)}.
			\end{equation*}
			If $p=\infty$,
			\begin{equation*}
			\left|\left|T_\varepsilon(\varphi)\right|\right|_{L^\infty\left((0,1)\times Y^*\right)}=\left|\left|\varphi\right|\right|_{L^\infty(R_0^\varepsilon)}\leq \left|\left|\varphi\right|\right|_{L^\infty(R^\varepsilon)};
			\end{equation*}
			\item $\forall\varphi \in W^{1,p}(R^\varepsilon)$, $1\leq p\leq \infty$,
			\begin{equation*}
			\partial_{y_1}T_\varepsilon(\varphi)=\varepsilon^\alpha T_\varepsilon(\partial_{x}\varphi)\mbox{ and }\partial_{y_2}T_\varepsilon(\varphi)=\varepsilon T_\varepsilon(\partial_{y}\varphi)\mbox{ a.e. in }(0,1)\times Y^*;
			\end{equation*}
			\item If $\varphi \in W^{1,p}(R^\varepsilon)$, then $T_\varepsilon(\varphi)  \in L^p\left((0,1);W^{1,p}(Y^*)\right)$, $1\leq p\leq \infty$. Besides, for $1\leq p< \infty$, we have 
			$$
			\begin{gathered}
				\left|\left|\partial_{y_1}T_\varepsilon(\varphi)\right|\right|_{L^p\left((0,1)\times Y^*\right)}=\varepsilon^\alpha\left(\frac{L}{\varepsilon}\right)^{\frac{1}{p}}\left|\left|\partial_{x}\varphi\right|\right|_{L^p(R_0^\varepsilon)}\leq \varepsilon^\alpha\left(\frac{L}{\varepsilon}\right)^{\frac{1}{p}}\left|\left|\partial_{x}\varphi\right|\right|_{L^p(R^\varepsilon)}\\
				\left|\left|\partial_{y_2}T_\varepsilon(\varphi)\right|\right|_{L^p\left((0,1)\times Y^*\right)}=\varepsilon\left(\frac{L}{\varepsilon}\right)^{\frac{1}{p}}\left|\left|\partial_{y}\varphi\right|\right|_{L^p(R_0^\varepsilon)}\leq \varepsilon\left(\frac{L}{\varepsilon}\right)^{\frac{1}{p}}\left|\left|\partial_{y}\varphi\right|\right|_{L^p(R^\varepsilon)}.
			\end{gathered}
			$$
			If $p=\infty$,
			\begin{eqnarray*}
				&&\left|\left|\partial_{y_1}T_\varepsilon(\varphi)\right|\right|_{L^\infty\left((0,1)\times Y^*\right)}=\varepsilon^\alpha\left|\left|\partial_{x}\varphi\right|\right|_{L^\infty(R_0^\varepsilon)}\leq \varepsilon^\alpha\left|\left|\partial_{x}\varphi\right|\right|_{L^\infty(R^\varepsilon)}\\
				&&\left|\left|\partial_{y_2}T_\varepsilon(\varphi)\right|\right|_{L^\infty\left((0,1)\times Y^*\right)}=\varepsilon\left|\left|\partial_{y}\varphi\right|\right|_{L^\infty(R_0^\varepsilon)}\leq \varepsilon\left|\left|\partial_{y}\varphi\right|\right|_{L^\infty(R^\varepsilon)}.
			\end{eqnarray*}
		\end{enumerate}
	\end{proposition}

		 Notice that, due to the order of the height of the thin domain the factor $1/\varepsilon$ appears in properties 5 and 6. Then, it makes sense to consider the following rescaled Lebesgue measure in the thin domains
		$$
		\rho_\varepsilon(\mathcal{O})=\dfrac{1}{\varepsilon}|\mathcal{O}|\mbox{, }\forall \mathcal{O}\subset R^\varepsilon,
		$$
		which is widely considered in works involving thin domains. As a matter of fact, from now on, we use the following rescaled norms in the thin open sets
		\begin{eqnarray*}
			&&\left|\left|\left|\varphi\right|\right|\right|_{L^p(R^\varepsilon)}=\varepsilon^{-1/p}\left|\left|\varphi\right|\right|_{L^p(R^\varepsilon)},\forall\varphi\in L^p(R^\varepsilon), 1\leq p< \infty,\\
			&&\left|\left|\left|\varphi\right|\right|\right|_{W^{1,p}(R^\varepsilon)}=\varepsilon^{-1/p}\left|\left|\varphi\right|\right|_{W^{1,p}(R^\varepsilon)}\forall\varphi\in W^{1,p}(R^\varepsilon), 1\leq p< \infty.
		\end{eqnarray*} 
		For completeness we may denote $\left|\left|\left|\varphi\right|\right|\right|_{L^\infty(R^\varepsilon)}=\left|\left|\varphi\right|\right|_{L^\infty(R^\varepsilon)}$.
		
		 From property 6, we have
		$$
		\left|\left|T_\varepsilon(\varphi)\right|\right|_{L^p\left((0,1)\times Y^*\right)}\leq L^{1/p}\left|\left|\left|\varphi\right|\right|\right|_{L^p(R^\varepsilon)}, \quad 1\leq p< \infty.
		$$
		
		 Property 5 of Proposition \ref{unfoldingproperties} will be essential to pass to the limit when dealing with solutions of differential equations because it will allow us to transform any integral over the thin domain (which depends on the parameter $\varepsilon$) into an integral over the fixed set $(0, 1) \times Y^*$. Notice that, in view of this property, we may say that the unfolding operator “almost preserves” the integral of the functions since the “integration defect” arises only from the unique cell which is not completely included in $R^\varepsilon$ and it is controlled by the integral on $R_1^\varepsilon$. 
		 
		 Therefore, an important concept for the unfolding method is the following property called unfolding criterion for integrals (u.c.i.).
	
	\begin{definition}
		A sequence $\left(\varphi_\varepsilon\right)$ satisfies the unfolding criterion for integrals (u.c.i) if
		\begin{equation*}
		\frac{1}{\varepsilon}\int_{R_1^\varepsilon}|\varphi_\varepsilon|dxdy\rightarrow 0.
		\end{equation*}
	\end{definition}
	
	\begin{proposition}
		Let $\left(\varphi_\varepsilon\right)$ be a sequence in $L^p(R^\varepsilon)$, $1<p\leq\infty$ with the norm $\left|\left|\left|\varphi_\varepsilon\right|\right|\right|_{L^p(R^\varepsilon)}$ uniformly bounded. 
		Then, $\left(\varphi_\varepsilon\right)$ satisfies the (u.c.i). 
		
		Furthermore, let $\left(\psi_\varepsilon\right)$ be a sequence in $L^{q}(R^\varepsilon)$, also with $\left|\left|\left|\psi_\varepsilon\right|\right|\right|_{L^{q}(R^\varepsilon)}$ uniformly bounded, $\frac{1}{p}+\frac{1}{q}=\frac{1}{r}$, with $r>1$. 
		Then, the product sequence $(\varphi_\varepsilon\psi_\varepsilon)$ satisfies (u.c.i).
	
		If we still take $\phi\in L^{p'}(0,1)$, then, the sequence $\varphi_\varepsilon\phi$ satifies (u.c.i).
	\end{proposition}	

	\begin{proposition}
		Let $\left(\varphi_\varepsilon\right)$ be a sequence in $L^p(R^\varepsilon)$, $1<p<\infty$ with $\left|\left|\left|\varphi_\varepsilon\right|\right|\right|_{L^p(R^\varepsilon)}$ uniformly bounded and let $\left(\psi_\varepsilon\right)$ be a sequence in $L^{p'}(R^\varepsilon)$ set as follows
		\begin{equation*}
		\psi_\varepsilon(x,y)=\psi\left(\frac{x}{\varepsilon^\alpha},\frac{y}{\varepsilon}\right),
		\end{equation*}
		where $\psi\in L^{p'}(Y^*)$. Then, $(\varphi_\varepsilon\psi_\varepsilon)$ satisfies (u.c.i).
	\end{proposition}

	Now, let us recall some convergence properties of the unfolding operator as $\varepsilon$ goes to zero.
	
	\begin{theorem}\label{convergencetheorem}
		For a measurable function $f$ on $Y^*$, $L$-periodic in its first variable and extended by periodicity to $\left\{ (x,y)\in \mathbb{R}^2:x\in \mathbb{R}, 0<y<g(x)\right\}$,  define the sequence $(f_\varepsilon)$ by
		\begin{equation*}
		f^\varepsilon(x,y)=f\left(\frac{x}{\varepsilon^\alpha},\frac{y}{\varepsilon}\right)
		\end{equation*}
		a.e. for 
		$$(x,y)\in\left\lbrace(x,y)\in \mathbb{R}^2:x\in \mathbb{R}, 0<y<\varepsilon g\left(\frac{x}{\varepsilon^\alpha}\right)\right\rbrace.$$
		
		Then 
		\begin{eqnarray*}
		T_\varepsilon f^\varepsilon|_{(0,1)}(x,y_1,y_2)=\left\{\begin{array}{ll}
		f(y_1,y_2)\mbox{, for }(x,y_1,y_2)\in I_\varepsilon\times Y^*,\\
		0\mbox{, for }(x,y_1,y_2)\in \Lambda_\varepsilon\times Y^*.
		\end{array}\right.
		\end{eqnarray*}
		
		Moreover, if $f\in L^p_\#(Y^*)$, then 
		\begin{equation*}
		T_\varepsilon f^\varepsilon\rightarrow f
		\end{equation*}
		strongly in $L^p_\#\left((0,1)\times Y^*\right)$.
	\end{theorem}
	\begin{proposition}\label{convergenceplaplacetype}
		Let $f\in L^p\left((0,1);L^p_\#(Y^*)\right)$ and extend it periodically in $y_1$-direction defining 
		\begin{equation} \label{eq448}
		f^\varepsilon(x,y):=f\left(x,\frac{x}{\varepsilon^\alpha},\frac{y}{\varepsilon}\right)\in L^p(R^\varepsilon).
		\end{equation}
		Then, 
		\begin{equation*}
		T_\varepsilon f^\varepsilon\rightarrow f \mbox{ strongly in }L^p\left((0,1)\times Y^*\right).
		\end{equation*}
	\end{proposition}
	\begin{proof}
		It follows from Theorem \ref{convergencetheorem} and the density of the tensor product 
		$C_0^\infty(0,1)\otimes L^{p}_\#(Y^*)$ in $L^p\left((0,1);L^p_\#(Y^*)\right)$.
	\end{proof}
	
	\begin{remark} \label{re483} Using Proposition \ref{convergenceplaplacetype}, we also have that, if 
	$T_\varepsilon f^\varepsilon\rightarrow f$ strongly in $L^p\left((0,1)\times Y^*\right)$, then 
	\begin{equation}
	T_\varepsilon \left(|f^\varepsilon|^{p-2}f^\varepsilon\right)\rightarrow |f|^{p-2}f\mbox{ strongly in }L^{p'}\left((0,1)\times Y^*\right).
	\end{equation}
	In particular, we can take $f^\varepsilon$ as in \eqref{eq448}.
	\end{remark}
	
	\begin{proposition}
		Let $\varphi\in L^p(0,1)$, $1\leq p<\infty$. Then, considering $\varphi$ as a function defined in $R^\varepsilon$, we have
		\begin{equation*}
		T_\varepsilon\varphi\rightarrow \varphi\mbox{ strongly in }L^p\left((0,1)\times Y^*\right).
		\end{equation*}
	\end{proposition}
	\begin{proposition}
		Let $(\varphi_\varepsilon)$ be a sequence in $L^p(0,1)$, $1\leq p<\infty$, such that 
		\begin{equation*}
		\varphi_\varepsilon\rightarrow \varphi\mbox{ strongly in }L^p(0,1).
		\end{equation*}
		Then, 
		\begin{equation*}
		T_\varepsilon\varphi_\varepsilon\rightarrow \varphi\mbox{ strongly in }L^p\left((0,1)\times Y^*\right).
		\end{equation*}
	\end{proposition}
	
	Next, we recall a convergence result which does not depend on the value of the parameter $\alpha$. To do that, we first introduce a suitable decomposition to functions $\varphi\in W^{1,p}(R^\varepsilon)$ where the geometry of the thin domains plays a crucial role. 
	We write $\varphi(x,y)=V(x)+\varphi_r(x,y)$ where $V$ is defined as follows
	\begin{equation}\label{Vauxiliar}
	V(x):=\frac{1}{\varepsilon g_0}\int_0^{\varepsilon g_0}\varphi(x,s) \, ds \quad \mbox{ a.e. } x\in (0,1).
	\end{equation}
	We set $\varphi_r(x,y)\equiv \varphi(x,y)-V(x)$.
	\begin{proposition}\label{propositionconvergence}
		Let $(\varphi_\varepsilon)\subset W^{1,p}(R^\varepsilon)$, $1<p<\infty$, with $\left|\left|\left|\varphi_\varepsilon\right|\right|\right|_{W^{1,p}(R^\varepsilon)}$ uniformly bounded and $V_\varepsilon(x)$ defined as in \eqref{Vauxiliar}. Then, there exists a function $\varphi\in W^{1,p}(0,1)$ such that, up to subsequences
		\begin{eqnarray*}
			&&V_\varepsilon\rightharpoonup \varphi\mbox{ weakly in } W^{1,p}(0,1)\mbox{ and strongly in } L^{p}(0,1),\\
			&&T_\varepsilon V_\varepsilon\rightarrow \varphi\mbox{ strongly in } L^{p}\left((0,1)\times Y^*\right),\\
			&&\left|\left|\left|\varphi_\varepsilon-V_\varepsilon\right|\right|\right|_{L^{p}(R^\varepsilon)}\rightarrow 0,\\
			&&\left|\left|\left|\varphi_\varepsilon-\varphi\right|\right|\right|_{L^{p}(R^\varepsilon)}\rightarrow 0,\\
			&&T_\varepsilon\varphi_\varepsilon \rightharpoonup \varphi\mbox{ weakly in } L^p\left((0,1);W^{1,p}(Y^*)\right),\\
			&&T_\varepsilon\varphi_\varepsilon \rightarrow \varphi \mbox{ strongly in } L^{p}\left((0,1)\times Y^*\right).
		\end{eqnarray*}
		
		Furthermore, there exists $\overline{\varphi}\in L^{p}\left((0,1)\times Y^*\right)$ with $\partial_{y_2}\overline{\varphi}\in L^{p}\left((0,1)\times Y^*\right)$ such that, up to subsequences
		\begin{eqnarray*}
			&&\frac{1}{\varepsilon}T_\varepsilon(\varphi_r^\varepsilon)\rightharpoonup\overline{\varphi}\mbox{ weakly in }L^{p}\left((0,1)\times Y^*\right), \textrm{ and } \\
			&&T_\varepsilon(\partial_{y}\varphi_\varepsilon)\rightharpoonup \partial_{y_2}\overline{\varphi}\mbox{ weakly in }L^{p}\left((0,1)\times Y^*\right)
		\end{eqnarray*}
		where $\varphi_r^\varepsilon\equiv\varphi_\varepsilon-V_\varepsilon$.
	\end{proposition}

Now, let us recall a compactness result which allows us to identify the limit of the image of the gradient of a uniformly bounded sequence by the unfolding operator method as $0<\alpha \leq 1$ in \eqref{TDs}. 
	
	\begin{theorem}\label{thmressonant} \label{thmalpha<1}
		Let $(\varphi_\varepsilon)\subset W^{1,p}(R^\varepsilon)$, $1<p<\infty$, with $\left|\left|\left|\varphi_\varepsilon\right|\right|\right|_{W^{1,p}(R^\varepsilon)}$ uniformly bounded. 
		
		Then, there exists $\varphi\in W^{1,p}(0,1)$ and $\varphi_1\in L^p((0,1);W^{1,p}_\#(Y^*))$ such that (up to a subsequence) 
		
		\begin{itemize} 
		\item[a)] if $\alpha=1$, we have
		\begin{eqnarray*}
			&&T_\varepsilon \varphi_\varepsilon\rightharpoonup \varphi\mbox{ weakly in } L^p\left((0,1);W^{1,p}(Y^*)\right),\\
			&&T_\varepsilon\partial_x\varphi_\varepsilon\rightharpoonup \partial_x \varphi+\partial_{y_1}\varphi_1\mbox{ weakly in }L^p\left((0,1)\times Y^*\right),\\
			&&T_\varepsilon\partial_y\varphi_\varepsilon\rightharpoonup \partial_{y_2}\varphi_1\mbox{ weakly in }L^p\left((0,1)\times Y^*\right).
		\end{eqnarray*}
	
	\item[b)] If $\alpha<1$, we obtain 
	 $\partial_{y_2}\varphi_1=0$  and 
		\begin{eqnarray*}
			&&T_\varepsilon\varphi_\varepsilon\rightharpoonup \varphi\mbox{, weakly in }L^p\left((0,1);W^{1,p}(Y^*)\right),\\
			&&T_\varepsilon\partial_x\varphi_\varepsilon\rightharpoonup \partial_x\varphi+\partial_{y_1}\varphi_1\mbox{, weakly in }L^p\left((0,1)\times Y^*\right).
		\end{eqnarray*}
			\end{itemize}
	
	\end{theorem}
	\begin{proof}
		See \cite[Theorem 3.1 and 4.1]{AM2} respectively.  
	\end{proof}

Finally, we obtain uniform boundedness to the solutions of the $p$-Laplacian problem \eqref{problem01} for any value of $\alpha>0$.

	\begin{proposition}\label{uniformlimtation}
		Consider the variational formulation of our problem:
		\begin{equation}\label{1001}
		\int_{R^\varepsilon} \left\{ \left|\nabla u_\varepsilon\right|^{p-2}\nabla u_\varepsilon\nabla\varphi+\left|u_\varepsilon\right|^{p-2} u_\varepsilon\varphi \right\} dxdy = \int_{R^\varepsilon}f^\varepsilon\varphi dxdy, \; \varphi\in W^{1,p}(R^\varepsilon),
		\end{equation}
		where $f^\varepsilon$ satisfies $$\left|\left|\left|f^\varepsilon\right|\right|\right|_{L^{p'}(R^\varepsilon)}\leq c$$ for some positive constant $c$ independent of $\varepsilon>0$.
		Then, 
		\begin{equation*}
		\begin{gathered}
			\left|\left|\left|u_\varepsilon\right|\right|\right|_{L^{p}(R^\varepsilon)}\leq c,\quad \quad 
			\left|\left|\left|u_\varepsilon\right|\right|\right|_{W^{1,p}(R^\varepsilon)}\leq c,\\
			\left|\left|\left|\left|\nabla u_\varepsilon\right|^{p-2}\nabla u_\varepsilon\right|\right|\right|_{L^{p'}(R^\varepsilon)}\leq c.
		\end{gathered}
		\end{equation*}
	\end{proposition}
	\begin{proof}
		Take $\varphi=u_\varepsilon$ in \eqref{1001}. Then,
		\begin{eqnarray*}
		\left|\left|u_\varepsilon\right|\right|_{W^{1,p}(R^\varepsilon)}^p  = 
		\int_{R^\varepsilon} \left\{ \left|\nabla u_\varepsilon\right|^p+\left| u_\varepsilon\right|^p \right\} dx dy
		\leq \left|\left|f^\varepsilon\right|\right|_{L^{p'}(R^\varepsilon)} \left|\left|u_\varepsilon\right|\right|_{L^p(R^\varepsilon)}.
		\end{eqnarray*}
		%
%		One can say that
%		\begin{equation}
%		\left|\left|u_\varepsilon\right|\right|_{W^{1,p}(R^\varepsilon)}^p\leq \left|\left|f^\varepsilon\right|\right|\left|\left|u_\varepsilon\right|\right|_{W^{1,p}(R^\varepsilon)}.
%		\end{equation}
		%
		Hence, 
		\begin{equation*}
		\left|\left|\left|u_\varepsilon\right|\right|\right|_{W^{1,p}(R^\varepsilon)}\leq c.
		\end{equation*}
		
		Therefore, the sequence $u_\varepsilon$ and $|\nabla u_\varepsilon|^{p-2}\nabla u_\varepsilon$, are respectively bounded in $L^p(R^\varepsilon)$ and $(L^{p'}(R^\varepsilon))^2$ under the norm $\left|\left|\left|\cdot\right|\right|\right|$.
	\end{proof}

   	\section{The resonant case: $\alpha=1$.} \label{seal1}
	
	In this section, we use the results on the Unfolding Operator described in Section \ref{USection}  in order to pass to the limit in problem \eqref{problem01} assuming $\alpha=1$. Notice that this case is called resonant since the amplitude and period of the oscillation are of the same order as the thickness of the thin domain. 
	
	Thus, we consider here in this section, the following two-dimensional thin domain family
	$$
	R^\varepsilon=\left\{(x,y)\in \mathbb{R}^2: 0<x<1, 0<y<\varepsilon g\left(\dfrac{x}{\varepsilon}\right)\right\}.
	$$
	with $g$ satisfying hypothesis ($\mathbf{H_g}$).
	We have the following result. 
		\begin{theorem}\label{resultressonant}
		Let $u_\varepsilon$ be the solution of problem \eqref{problem01} with $f^\varepsilon$ satisfying 
		$$\left|\left|\left|f^\varepsilon\right|\right|\right|_{L^{p'}(R^\varepsilon)}\leq c$$
		 for $c>0$ independent of $\varepsilon>0$. Suppose also that 
		\begin{equation}  \label{fhat01}
		T_\varepsilon f^\varepsilon\rightharpoonup\hat{f}\mbox{ weakly in }L^{p'}\left((0,1)\times Y^*\right).
		\end{equation}
		
		Then, there exists $(u,u_1)\in W^{1,p}(0,1)\times L^p((0,1);W^{1,p}_\#(Y^*))$ such that
		\begin{eqnarray*}
			\left\lbrace \begin{array}{llll}
				T_\varepsilon u_\varepsilon\rightharpoonup u\mbox{ weakly in } L^p((0,1);W^{1,p}(Y^*)),\\
				T_\varepsilon \left(\partial_x u_\varepsilon\right)\rightharpoonup \partial_xu+\partial_{y_1}u_1(x,y_1,y_2)\mbox{ weakly in } L^p\left((0,1);W^{1,p}(Y^*)\right),\\
				T_\varepsilon \left(\partial_y u_\varepsilon\right)\rightharpoonup \partial_{y_2}u_1(x,y_1,y_2)\mbox{ weakly in } L^p\left((0,1);W^{1,p}(Y^*)\right),
			\end{array}\right.
		\end{eqnarray*}
		and $u$ is the solution of the problem
		\begin{equation}\label{homogenizedlimitalfa=1}
		\int_0^1\left(q\,|\partial_xu|^{p-2}\partial_xu\partial_x \varphi+|\partial_xu|^{p-2}\partial_xu\varphi \right)dx =\int_0^1\bar{f}\varphi \,dx ,
		\end{equation}
		where 
		\begin{eqnarray}\label{falpha1}
		\bar{f}=\frac{1}{|Y^*|}\int_{Y^*}\hat{f} \, dy_1dy_2,
		\end{eqnarray}
		$$
		q=\frac{1}{|Y^*|}\int_{Y^*}|\nabla v|^{p-2}\partial_{y_1} v \, dy_1dy_2,
		$$
		and $v$ is the solution of the auxiliar problem
		\begin{eqnarray}\label{auxiliarproblem}
		\begin{gathered}
		\displaystyle\int_{Y^*}\left|\nabla v\right|^{p-2}\nabla v\nabla \varphi dy_1dy_2=0, \quad \forall\varphi\in W^{1,p}_{\#,0}(Y^*),\\
		(v-y_1)\in  W^{1,p}_{\#,0}(Y^*),
		\end{gathered}
		\end{eqnarray} 
		where $W^{1,p}_{\#,0}(Y^*)$ denotes the subspace of $W^{1,p}_{\#}(Y^*)$ of functions with zero average.
		
		Moreover, 
		\begin{equation} \label{eq680}
		u'(x)\nabla_{y}v(y_1,y_2)=(u'(x),0)+\nabla_{y} u_1(x,y_1,y_2),
		\end{equation}
		where $\nabla_y\cdot=\left(\partial_{y_1}\cdot,\partial_{y_2}\cdot\right)$.
	\end{theorem}
	\begin{proof}
		From Proposition \ref{unfoldingproperties}, we can rewrite \eqref{1001} as
		\begin{equation} \label{variationalunfolded}
		\begin{gathered}
		\int_{(0,1)\times Y^*}T_\varepsilon\left(\left|\nabla u_\varepsilon\right|^{p-2}\nabla u_\varepsilon\right)T_\varepsilon\nabla\varphi dxdy_1dy_2+\frac{L}{\varepsilon}\int_{R_1^\varepsilon}\left|\nabla u_\varepsilon\right|^{p-2}\nabla u_\varepsilon\nabla\varphi dxdy\\
		+\int_{(0,1)\times Y^*}T_\varepsilon\left(\left| u_\varepsilon\right|^{p-2} u_\varepsilon\right)T_\varepsilon\varphi dxdy_1dy_2+\frac{L}{\varepsilon}\int_{R_1^\varepsilon}\left| u_\varepsilon\right|^{p-2} u_\varepsilon\varphi dxdy\\
		= \int_{(0,1)\times Y^*}T_\varepsilon f^\varepsilon T_\varepsilon \varphi dxdy_1dy_2+\frac{L}{\varepsilon}\int_{R_1^\varepsilon}f^\varepsilon\varphi dxdy.
		\end{gathered}
		\end{equation}
		By Proposition \ref{uniformlimtation} and Theorem \ref{thmressonant}, there exist 
		$$u\in W^{1,p}(0,1), \quad u_1\in L^p((0,1);W^{1,p}_\#(Y^*)) \quad \textrm{ and } \quad a_0\in L^p\left((0,1)\times Y^*\right)^2$$ 
		such that, up to subsequences,
		\begin{eqnarray}\label{unfoldingconvergence}
		\left\lbrace \begin{array}{llll}
		T_\varepsilon u_\varepsilon\rightarrow u\mbox{ strongly in } L^p\left((0,1)\times Y^*\right),\\
		T_\varepsilon \left(\partial_x u_\varepsilon\right)\rightharpoonup \partial_xu+\partial_{y_1}u_1(x,y_1,y_2)\mbox{ weakly in } L^p\left((0,1);W^{1,p}(Y^*)\right),\\
		T_\varepsilon \left(\partial_y u_\varepsilon\right)\rightharpoonup \partial_{y_2}u_1(x,y_1,y_2)\mbox{ weakly in } L^p\left((0,1);W^{1,p}(Y^*)\right),\\
		T_\varepsilon\left(\left|\nabla u_\varepsilon\right|^{p-2}\nabla u_\varepsilon\right)\rightharpoonup a_0\mbox{ weakly in } L^p\left((0,1)\times Y^*\right)^2.
		\end{array}\right.
		\end{eqnarray}
		
		We still have from Remark \ref{re483} that
		$$
		\left|T_\varepsilon u_\varepsilon\right|^{p-2}T_\varepsilon u_\varepsilon\rightarrow |u|^{p-2}u\mbox{ strongly in } L^{p'}\left((0,1)\times Y^*\right).
		$$
		Hence, passing to the limit in \eqref{variationalunfolded} for test functions $\varphi\in W^{1,p}(0,1)$, we get
		\begin{equation}\label{prehomogenized}
		\int_{(0,1)\times Y^*}a_0\nabla\varphi +\left|u\right|^{p-2}u\varphi dxdy_1dy_2=\int_{(0,1)\times Y^*} \hat{f}\varphi dxdy_1dy_2.
		\end{equation}
		
		Let $\phi\in C_0^\infty(0,1)$ and $\psi\in W^{1,p}_\#(Y^*)$. Define the sequence
		\begin{equation*}
		v_\varepsilon(x,y)=\varepsilon\phi(x)\psi\left(\frac{x}{\varepsilon},\frac{y}{\varepsilon}\right).
		\end{equation*}
		We apply the unfolding operator in this sequence and obtain
		\begin{eqnarray*}
		T_\varepsilon v_\varepsilon\rightarrow 0\mbox{ strongly in }L^p\left((0,1)\times Y^*\right),\\
		T_\varepsilon(\partial_x v_\varepsilon)\rightarrow \phi \partial_{y_1}\psi \mbox{ strongly in }L^p\left((0,1)\times Y^*\right),\\
		T_\varepsilon(\partial_y v_\varepsilon)\rightarrow \phi \partial_{y_2}\psi \mbox{ strongly in }L^p\left((0,1)\times Y^*\right).
		\end{eqnarray*}
		
		Thus, taking $v_\varepsilon$ as a test function in \eqref{variationalunfolded}, we obtain at $\varepsilon=0$ that 
		\begin{eqnarray}\label{preaxiliar}
		\int_{(0,1)\times Y^*}a_0\phi(x)\nabla_y\psi dxdy_1dy_2=0.
		\end{eqnarray}
		
		Hence, we get from \eqref{preaxiliar} and the density of the tensor product 
		$C_0^\infty(0,1)\otimes W^{1,p}_\#(Y^*)$ in $L^p((0,1);W^{1,p}_\#(Y^*))$ that
		\begin{eqnarray}\label{preaxiliar01}
		\int_{(0,1)\times Y^*}a_0 \, \nabla_y\psi \, dxdy_1dy_2=0, \quad \forall\psi\in L^p((0,1);W^{1,p}_\#(Y^*)).
		\end{eqnarray}

		Now, let us identify $a_0$. For this sake, let $u_1\in L^p((0,1);W^{1,p}_\#(Y^*))$ and $u\in W^{1,p}(0,1)$ be the ones obtained in \eqref{unfoldingconvergence}. Extend $\nabla_y u_1$ in the $y_1$-direction, and then define the sequence
		\begin{equation} \label{eqW_}
		W_\varepsilon(x,y)=\left(\partial_xu(x),0\right)+\nabla_y u_1\left(x,\frac{x}{\varepsilon},\frac{y}{\varepsilon}\right).
		\end{equation}

		See that $W_\varepsilon\in L^p(R^\varepsilon)\times L^p(R^\varepsilon)$. Also, due to Proposition \ref{convergenceplaplacetype}, we have
		\begin{equation}\label{convWeps}
		\begin{gathered}
		T_\varepsilon W_\varepsilon\rightarrow \left(\partial_xu,0\right)+\nabla_y u_1, \quad \textrm{ and } \\
		T_\varepsilon\left(|W_\varepsilon|^{p-2}W_\varepsilon\right)\rightarrow \left|\left(\partial_xu,0\right)+\nabla_y u_1\right|^{p-2}\left[\left(\partial_xu,0\right)+\nabla_y u_1\right]   
		\end{gathered}
		\end{equation}
		strongly in $L^p\left((0,1)\times Y^*\right)^2$ and $L^{p'}\left((0,1)\times Y^*\right)^2$, respectively.

		We need to prove that
		\begin{equation*}
		\int_{(0,1)\times Y^*}\left[\left|\nabla u+\nabla_yu_1\right|^{p-2}\left(\nabla u+\nabla_yu_1\right)-a_0\right]\varphi \, dxdy_1dy_2=0,
		\end{equation*}
		for all $\varphi\in C_0^\infty\left((0,1)\times Y^*\right) \times C_0^\infty\left((0,1)\times Y^*\right)$ in order to identify $a_0$.
		
		For this sake, let us first prove that the right hand side of 
			\begin{equation} \label{693}
			0\leq\int_{(0,1)\times Y^*}T_\varepsilon\left[|\nabla u_\varepsilon|^{p-2}\nabla u_\varepsilon-|W_\varepsilon|^{p-2}W_\varepsilon\right]T_\varepsilon\left(\nabla u_\varepsilon-W_\varepsilon\right)dxdy_1dy_2
			\end{equation}
			converges to zero. Notice that by the monotonicity of $|\cdot|^{p-2}\cdot$ (see Proposition \ref{proposicaoplaplaciano}) the inequality above is obtained. To pass to the limit in \eqref{693}, we first use a distributive in the integral. 
			
			Using \eqref{variationalunfolded}, denoting $dY=dy_1 dy_2$ and using \eqref{unfoldingconvergence}, we get that 
			\begin{equation*}
				\begin{gathered}
				\lim_{\varepsilon\rightarrow 0} \int_{(0,1)\times Y^*}T_\varepsilon\left(|\nabla u_\varepsilon|^{p-2}\nabla u_\varepsilon\right)T_\varepsilon(\nabla u_\varepsilon)dxdY  \\
				 = \lim_{\varepsilon\rightarrow 0}\left[\int_{(0,1)\times Y^*}T_\varepsilon f^\varepsilon T_\varepsilon u_\varepsilon dxdY+\dfrac{L}{\varepsilon}\int_{R^\varepsilon_1}f^\varepsilon u_\varepsilon dxdy \right.\\
				\left. - \int_{(0,1)\times Y^*}T_\varepsilon\left(| u_\varepsilon|^{p-2} u_\varepsilon\right)T_\varepsilon u_\varepsilon dxdY -\dfrac{L}{\varepsilon}\int_{R^\varepsilon_1}|u_\varepsilon|^{p-2} u_\varepsilon u_\varepsilon dxdy\right]=\\
				\int_{(0,1)\times Y^*}\left(\hat{f}-|u|^{p-2}u\right)udxdY.
				\end{gathered}
			\end{equation*}
			Consequently, we get from \eqref{prehomogenized} that 
			\begin{equation} \label{eq777}
			\lim_{\varepsilon\rightarrow 0} \int_{(0,1)\times Y^*}T_\varepsilon\left(|\nabla u_\varepsilon|^{p-2}\nabla u_\varepsilon\right)T_\varepsilon(\nabla u_\varepsilon)dxdY =
				\int_{(0,1)\times Y^*} a_0 \nabla u dxdY.
			\end{equation}
			
			On the other hand, due to \eqref{unfoldingconvergence}, \eqref{convWeps} and \eqref{preaxiliar01}, we get
			\begin{equation}  \label{eq780}
			\begin{gathered}
				\lim_{\varepsilon\rightarrow 0} \int_{(0,1)\times Y^*}T_\varepsilon\left(|\nabla u_\varepsilon|^{p-2}\nabla u_\varepsilon\right)T_\varepsilon(W_\varepsilon)dxdY \\ = \int_{(0,1)\times Y^*} a_0 \left(\nabla u+\nabla_yu_1\right)  dxdY\\
				=\int_{(0,1)\times Y^*} a_0\nabla u dxdY.
			\end{gathered}
			\end{equation}

			Finally, we have 			
			\begin{equation} \label{eq790}
			\begin{gathered}
			\lim_{\varepsilon\rightarrow 0} \int_{(0,1)\times Y^*}T_\varepsilon\left(|W_\varepsilon|^{p-2}W_\varepsilon\right)T_\varepsilon\left(\nabla u_\varepsilon-W_\varepsilon\right)dxdY = 0,
			\end{gathered}
			\end{equation}
			by \eqref{convWeps} and \eqref{unfoldingconvergence}. 
			Indeed, we have $T_\varepsilon (\nabla u_\varepsilon - W_\varepsilon) \rightharpoonup 0$ weakly in $L^p((0,1) \times Y^*)$.
			
			Thus, from \eqref{eq777}, \eqref{eq780} and \eqref{eq790}, we can pass to the limit in \eqref{693} to get
			\begin{equation}\label{694}
				\int_{(0,1)\times Y^*}T_\varepsilon\left[|\nabla u_\varepsilon|^{p-2}\nabla u_\varepsilon-|W_\varepsilon|^{p-2}W_\varepsilon\right]T_\varepsilon\left(\nabla u_\varepsilon-W_\varepsilon\right)dxdY\rightarrow 0.
			\end{equation}
			
			Suppose $p\geq 2$. Then, by Proposition \ref{proposicaoplaplaciano}, we have 
			\begin{eqnarray} \label{eq823}
			&&\int_{(0,1)\times Y^*}\left|T_\varepsilon \nabla u_\varepsilon-T_\varepsilon W_\varepsilon\right|^p dxdY\nonumber\\
			&\leq&c\int_{(0,1)\times Y^*}T_\varepsilon\left(|\nabla u_\varepsilon|^{p-2}\nabla u_\varepsilon-|W_\varepsilon|^{p-2}W_\varepsilon\right)\left(T_\varepsilon \nabla u_\varepsilon-T_\varepsilon W_\varepsilon\right) dxdY\nonumber\\
			&\rightarrow& 0 \qquad \mbox{ as }\varepsilon\rightarrow 0.
			\end{eqnarray}
			If $1<p\leq2$, we have, using a H\"older's inequality for the exponent $\frac{2}{p}$ (and its conjugate $\frac{2}{2-p}$) and Proposition \ref{proposicaoplaplaciano}, 
			\begin{eqnarray}\label{824}
			&&\int_{(0,1)\times Y^*}\left|T_\varepsilon \nabla u_\varepsilon-T_\varepsilon W_\varepsilon\right|^p dxdY\nonumber\\
			&=&\int_{(0,1)\times Y^*}\left|T_\varepsilon \nabla u_\varepsilon-T_\varepsilon W_\varepsilon\right|^p \frac{\left(1+|T_\varepsilon\nabla u_\varepsilon|+|T_\varepsilon W_\varepsilon|\right)^{\frac{(p-2)p}{2}}}{\left(1+|T_\varepsilon\nabla u_\varepsilon|+|T_\varepsilon W_\varepsilon|\right)^{\frac{(p-2)p}{2}}} dxdY\nonumber\\
			&\leq& \left(\int_{(0,1)\times Y^*}\left|T_\varepsilon \nabla u_\varepsilon-T_\varepsilon W_\varepsilon\right|^2\left(1+|T_\varepsilon\nabla u_\varepsilon|+|T_\varepsilon W_\varepsilon|\right)^{(p-2)}dxdY\right)^{p/2}\nonumber\\
			&&\qquad\qquad\qquad\qquad\left(\int_{(0,1)\times Y^*}\left(1+|T_\varepsilon\nabla u_\varepsilon|+|T_\varepsilon W_\varepsilon|\right)^pdxdY\right)^{(2-p)/2}\nonumber\\
			&\leq& c \left[\int_{(0,1)\times Y^*}T_\varepsilon\left(|\nabla u_\varepsilon|^{p-2}\nabla u_\varepsilon-|W_\varepsilon|^{p-2}W_\varepsilon\right)\left(T_\varepsilon \nabla u_\varepsilon-T_\varepsilon W_\varepsilon\right) dxdY\right]^{p/2}\nonumber\\
			&\to& 0.
			\end{eqnarray}
		Now, let us prove 
		\begin{equation} \label{eq829}
		\int_{(0,1)\times Y^*}T_\varepsilon\left(|\nabla u_\varepsilon|^{p-2}\nabla u_\varepsilon-|W_\varepsilon|^{p-2}W_\varepsilon\right)\varphi \, dxdY \rightarrow 0
		\end{equation}
		for any test function $\varphi\in C_0^\infty\left((0,1)\times Y^*\right)\times C_0^\infty\left((0,1)\times Y^*\right)$.
		
		First, consider $p\geq 2$. Therefore, from Corollary \eqref{corolarioplaplaciano} and Young inequality, we get
		\begin{eqnarray*}
		&&\nonumber\int_{(0,1)\times Y^*}T_\varepsilon\left(|\nabla u_\varepsilon|^{p-2}\nabla u_\varepsilon-|W_\varepsilon|^{p-2}W_\varepsilon\right)\varphi dxdY\\\nonumber
		&\leq&c\int_{(0,1)\times Y^*}\left||T_\varepsilon\nabla u_\varepsilon|^{p-2}T_\varepsilon\nabla u_\varepsilon-|T_\varepsilon W_\varepsilon|^{p-2}T_\varepsilon W_\varepsilon\right|dxdY\\\nonumber
		&\leq& c\int_{(0,1)\times Y^*}\left(1+|T_\varepsilon \nabla u_\varepsilon|+|T_\varepsilon W_\varepsilon|\right)^{p-2}\left|T_\varepsilon \left(\nabla u_\varepsilon-W_\varepsilon\right)\right|dxdY\\\nonumber
		&\leq&c\left(\int_{(0,1)\times Y^*}\left(1+|T_\varepsilon \nabla u_\varepsilon|+|T_\varepsilon W_\varepsilon|\right)^{p}dxdY\right)^{1/p'}\\
		&&\qquad\qquad\left(\int_{(0,1)\times Y^*}\left|T_\varepsilon \nabla u_\varepsilon-T_\varepsilon W_\varepsilon\right|^p dxdY\right)^{1/p}\nonumber\\
		&\to& 0,
		\end{eqnarray*}
		by convergence \ref{eq823}.
		%
%	Using Proposition \eqref{proposicaoplaplaciano}, we have by \eqref{694} that the last integral in \eqref{aob} satisfies  
%		\begin{eqnarray} \label{eq823}
%		&&\int_{(0,1)\times Y^*}\left|T_\varepsilon \nabla u_\varepsilon-T_\varepsilon W_\varepsilon\right|^p dxdY\nonumber\\
%		&\leq&\int_{(0,1)\times Y^*}T_\varepsilon\left(|\nabla u_\varepsilon|^{p-2}\nabla u_\varepsilon-|W_\varepsilon|^{p-2}W_\varepsilon\right)\left(T_\varepsilon \nabla u_\varepsilon-T_\varepsilon W_\varepsilon\right) dxdY\nonumber\\
%		&\rightarrow& 0 \qquad \mbox{ as }\varepsilon\rightarrow 0,
%		\end{eqnarray}
		%		 
%		and then, we can pass to the limit in \eqref{aob} to get
%		\begin{equation}
%		\int_{(0,1)\times Y^*}\left(a_0-\left|\nabla u+\nabla_yu_1\right|^{p-2}\left(\nabla u+\nabla_yu_1\right)\right)\varphi dxdY\leq c\delta^{p'}
%		\end{equation}
%		for any $\delta>0$. Thus, for $p\geq 2$
%		\begin{equation} \label{eq830}
%		a_0=\left|\nabla u+\nabla_yu_1\right|^{p-2}\left(\nabla u+\nabla_yu_1\right)\mbox{ a.e. in }(0,1)\times Y^*.
%		\end{equation}
%		

		For $1<p<2$, we perform analogous arguments. Using Corollary \ref{corolarioplaplaciano}, a H\"older's inequality and the convergence \eqref{824}, one gets 
		\begin{eqnarray*}
		&&\nonumber\int_{(0,1)\times Y^*}T_\varepsilon\left(|\nabla u_\varepsilon|^{p-2}\nabla u_\varepsilon-|W_\varepsilon|^{p-2}W_\varepsilon\right)\varphi dxdY\\\nonumber
		&\leq&c\int_{(0,1)\times Y^*}\left|T_\varepsilon\nabla u_\varepsilon-T_\varepsilon W_\varepsilon\right|^{p-1}dxdY\\
		&=&\nonumber c\left(\int_{(0,1)\times Y^*}\left|T_\varepsilon\nabla u_\varepsilon-T_\varepsilon W_\varepsilon\right|^{p}dxdY\right)^{1/p'}\\
		&\to& 0.
		\end{eqnarray*}		
		Thus, for any $p> 1$ and $\varphi\in C_0^\infty\left((0,1)\times Y^*\right)\times C_0^\infty\left((0,1)\times Y^*\right)$, we have 
		$$
		\int_{(0,1)\times Y^*}\left(a_0-\left|\nabla u+\nabla_yu_1\right|^{p-2}\left(\nabla u+\nabla_yu_1\right)\right)\varphi dxdY=0.
		$$
%		and then, by density,  we have for any $\varphi\in L^{p}\left((0,1)\times Y^*\right)\times L^{p}\left((0,1)\times Y^*\right)$
%		\begin{equation*}
%		\int_{(0,1)\times Y^*}\left(a_0-\left|\nabla u+\nabla_yu_1\right|^{p-2}\left(\nabla u+\nabla_yu_1\right)\right)\varphi dxdY=0
%		\end{equation*}
%	getting \eqref{secondpartauxilar} for all $p \in (1,\infty)$.
		
		Finally, let us associate $a_0$ with the auxiliary problem \eqref{auxiliarproblem}. 
		We first rewrite \eqref{preaxiliar01} as 
		\begin{eqnarray}\label{preaxiliar02}
		\int_{(0,1)\times Y^*}\left|\nabla u+\nabla_yu_1\right|^{p-2}\left(\nabla u+\nabla_yu_1\right)\nabla_y\psi \, dxdY=0,
		\end{eqnarray}
		for any $\psi \in L^p((0,1);W^{1,p}_\#(Y^*))$.

		From Minty-Browder Theorem, we have that \eqref{preaxiliar02} sets a well posed problem, and then, for each $u \in W^{1,p}(0,1)$, possesses an unique solution $u_1 \in L^p((0,1);W^{1,p}_\#(Y^*)/\mathbb{R})$. Notice that $W^{1,p}_\#(Y^*)/\mathbb{R}$ is identified with the closed subspace of $W^{1,p}_\#(Y^*)$ consisting of all its functions with zero average.

		We might mention that multiplying the solution $v$ of the equation \eqref{auxiliarproblem} by $u'$, then $u'v$ depends on $x$ and belongs to the space $L^p((0,1);W^{1,p}_\#(Y^*)/\mathbb{R})$. Multiplying \eqref{auxiliarproblem} by a function $|\partial_xu|^{p-2}\partial_xu$ and by $\phi\in C_0^\infty(0,1)$, and integrating in $(0, 1)$,  we get
		\begin{equation*}
		\int_{(0,1)\times Y^*} \phi \, \left|\partial_xu\nabla_yv\right|^{p-2}\partial_xu\nabla_yv \,  \nabla_y\varphi \, dxdY=0, \quad  
		\forall\varphi\in W^{1,p}_\#(Y^*)/\mathbb{R}.
		\end{equation*}
		Thus, from the density of tensor product $C_0^\infty(0,1)\otimes (W^{1,p}_\#(Y^*)/\mathbb{R})$, we get
		\begin{equation}\label{preauxiliar03}
		\int_{(0,1)\times Y^*}\left|\partial_xu\nabla_yv\right|^{p-2}\partial_xu\nabla_yv\nabla_y\psi dxdY=0, \quad \forall\psi\in L^p((0,1);W^{1,p}_\#(Y^*)/\mathbb{R}).
		\end{equation}
		
		Hence, we can conclude \eqref{eq680} from equations \eqref{preaxiliar02} and \eqref{preauxiliar03}.
		Moreover, we can rewrite \eqref{prehomogenized} as 
$$
		\int_{(0,1)\times Y^*}\left|\partial_xu\nabla_yv\right|^{p-2}\partial_xu\nabla_yv\nabla\varphi +\left|u\right|^{p-2}u\varphi dxdY=\int_{(0,1)\times Y^*} \hat{f}\varphi dxdY,
$$
		which is equivalent to
		\begin{equation*}
		\int_0^1 q|\partial_xu|^{p-2}\partial_xu\partial_x\varphi+\left|u\right|^{p-2}u\varphi dx=\int_0^1 \bar{f}\varphi dx,
		\end{equation*}
		where $q$ is that one given by
		$$
		q=\int_{Y^*}|\nabla_y v|^{p-2}\partial_{y_1}v dY
		$$
		and $$\bar{f}=\frac{1}{|Y^*|}\int_{Y^*}\hat{f}dY.$$
		
		We point out that $q>0$. Indeed, since $\left(\nabla_yv -(1,0)\right)\in W^{1,p}_{\#,0}(Y^*)$, we get
		\begin{equation}\label{q>0}
		0<\int_{Y^*}|\nabla_y v|^{p}dY=\int_{Y^*}|\nabla_y v|^{p-2}\nabla_y v \left[(1,0) +\nabla_yv -(1,0)\right] dY=\int_{Y^*}|\nabla_y v|^{p-2}\partial_{y_1}v dY.
		\end{equation}
		Finally, problem \eqref{homogenizedlimitalfa=1} is well posed, we conclude the proof noting that $T_\varepsilon u_\varepsilon$ is a convergent sequence.
	\end{proof}	
	
	In the proof of Theorem \ref{resultressonant}, we have already obtained a corrector result. The corrector function is given by $W_\varepsilon$.
	Indeed, according to \cite{MasoA}, since we already have, by Propositions \ref{uniformlimtation} and \ref{propositionconvergence}, 
	$$||| u_\varepsilon - u|||_{L^p(R^\varepsilon)} \to 0, \quad \textrm{ as } \varepsilon \to 0,$$  
	we just need to construct the corrector to the term $\nabla u_\varepsilon$.  
	We have the following result.

	\begin{corollary} \label{cor1045}
		The function $W_\varepsilon$ defined in \eqref{eqW_} satisfies
		$$
		\left|\left|T_\varepsilon \nabla u_\varepsilon-T_\varepsilon W_\varepsilon\right|\right|_{L^p\left((0,1)\times Y^*\right)}\rightarrow 0.
		$$
	\end{corollary}
	\begin{proof}
		The proof follows from \eqref{eq823}, when $p\geq 2$, and \eqref{824}, when $1<p\leq 2$.
	\end{proof}

	\begin{remark}
	From Corollary \ref{cor1045} and Proposition 2.17 of \cite{AM2}, we can conclude that 
	$$
	|||\nabla u_\varepsilon-W_\varepsilon  |||_{L^p(R^\varepsilon)}\rightarrow 0, \quad \textrm{ as } \varepsilon \to 0.
	$$
	Consequently, due to \eqref{eqW_}, we obtain that the function 
	$$
	\nabla_y u_1\left(x, \frac{x}{\varepsilon}, \frac{y}{\varepsilon}\right), \quad (x,y) \in R^\varepsilon,
	$$
	works as a corrector function to $\nabla u_\varepsilon$ in $L^p(R^\varepsilon)$ with the norm $||| \cdot |||$ since it allows strong convergence.
	\end{remark}

	\section{Weak oscillation case: $\alpha<1$.}  \label{seal<1}
	
	Now, let us consider $\alpha<1$. Here, we deal with the oscillatory thin domain $R^\varepsilon$ given by
	$$
	R^\varepsilon=\left\{(x,y)\in \mathbb{R}^2: 0<x<1, 0<y<\varepsilon g\left(\dfrac{x}{\varepsilon^\alpha}\right)\right\}
	$$
	with $g$ satisfying hypothesis ($\mathbf{H_g}$) at Section 2.
	
	In order to obtain the homogenized equation, we will proceed as in the previous case $\alpha=1$, Section \ref{seal1}. 
	We show the following result.
	
	\begin{theorem}\label{weakresult}
		Let $u_\varepsilon$ be the solution of problem \eqref{problem01} with $f^\varepsilon$ satisfying $\left|\left|\left|f^\varepsilon\right|\right|\right|_{L^{p'}(R^\varepsilon)}\leq c$ for some positive constant independent of $\varepsilon>0$. Suppose also that 
		\begin{equation*}
		T_\varepsilon f^\varepsilon\rightharpoonup\hat{f}\mbox{ weakly in }L^{p'}\left((0,1)\times Y^*\right).
		\end{equation*}
		Then, there exists $(u,u_1)\in W^{1,p}(0,1)\times L^p\left((0,1);W^{1,p}_\#(Y^*)\right)$ with $\partial_{y_2}u_1=0$ such that
		\begin{eqnarray*}
		T_\varepsilon u_\varepsilon\rightharpoonup u \mbox{ weakly in }L^p\left((0,1);W^{1,p}(Y^*)\right)\\
		T_\varepsilon \partial_x u_\varepsilon\rightharpoonup \partial_x u+\partial_{y_1}u_1\mbox{ weakly in }L^p\left((0,1)\times Y^*\right)
		\end{eqnarray*}
		where $u$ is the unique solution of 
		\begin{equation*}\label{homogenizedequationalfa<1}
		\int_0^1\left[\dfrac{1}{\langle g\rangle_{(0,L)}\left\langle 1/g^{p'-1}\right\rangle_{(0,L)}^{p-1}}|\partial_x u|^{p-2}\partial_x u\partial_x \varphi +|u|^{p-2}u\varphi \right]dx=\int_0^1\bar{f}\varphi dx, \quad \forall \varphi\in W^{1,p}(0,1),
		\end{equation*}
		and 
$$		
\bar{f}=\dfrac{1}{|Y^*|}\int_{Y^*}\hat{f}dy_1dy_2.
$$
	\end{theorem}
	\begin{proof}
		From Proposition \ref{uniformlimtation} and Theorem \eqref{thmalpha<1}, there exist $u\in W^{1,p}(0,1)$ and $u_1\in L^p((0,1);W^{1,p}_\#(Y^*))$ with $\partial_{y_2}u_1=0$ such that, up to subsequences,
		\begin{eqnarray*}
		T_\varepsilon u_\varepsilon \rightharpoonup u\mbox{ weakly in } L^p\left((0,1);W^{1,p}(Y^*)\right),\\
		\label{1234}T_\varepsilon (|u_\varepsilon|^{p-2}u_\varepsilon)\rightarrow |u|^{p-2}u\mbox{ strongly in }L^p\left((0,1)\times Y^*\right),\\
		T_\varepsilon\partial_x u_\varepsilon\rightharpoonup \partial_x u+\partial_{y_1}u_1 \mbox{ weakly in }L^p\left((0,1)\times Y^*\right),\\
		T_\varepsilon\left(\left|\nabla u_\varepsilon\right|^{p-2}\nabla u_\varepsilon\right)\rightharpoonup a_0 \mbox{ weakly in }L^{p'}\left((0,1)\times Y^*\right)^2.
		\end{eqnarray*}
				
		We rewrite \eqref{variationalproblem01} and obtain as in \eqref{variationalunfolded} that 
		\begin{eqnarray}
			&&\nonumber\int_{(0,1\times Y^*)}T_\varepsilon\left(\left|\nabla u_\varepsilon\right|^{p-2}\nabla u_\varepsilon\right)T_\varepsilon \nabla \varphi dxdy_1dy_2+\frac{L}{\varepsilon}\int_{R_1^\varepsilon}\left|\nabla u_\varepsilon\right|^{p-2}\nabla u_\varepsilon\nabla\varphi dxdy \nonumber \\
			&&+\nonumber\int_{(0,1\times Y^*)}T_\varepsilon\left(\left| u_\varepsilon\right|^{p-2} u_\varepsilon\right)T_\varepsilon\varphi dxdy_1dy_2+\frac{L}{\varepsilon}\int_{R_1^\varepsilon}\left| u_\varepsilon\right|^{p-2} u_\varepsilon\varphi dxdy \nonumber\\
			&&=\int_{(0,1\times Y^*)}T_\varepsilon f^\varepsilon T_\varepsilon \varphi dxdy_1dy_2+\frac{L}{\varepsilon}\int_{R_1^\varepsilon} f^\varepsilon\varphi dxdy, \quad \forall \varphi \in W^{1,p}(0,1).  \label{eq1147}
		\end{eqnarray}
		Hence, we can pass to the limit to get 
		\begin{eqnarray}\label{prehomogenizedalfa<1}
		&&\nonumber\int_{(0,1)\times Y^*}a_0 \nabla \varphi dxdy_1dy_2 +\int_{(0,1)\times Y^*}|u|^{p-2}u\varphi dxdy_1dy_2 \\
		&& \qquad \qquad =\int_{(0,1)\times Y^*}\hat{f}\varphi dxdy_1dy_2.
		\end{eqnarray}
		
		Now, let $\phi\in C_0^\infty(0,1)$ and $\psi\in W^{1,p}_\#(Y^*)$ with $\partial_{y_2}\psi=0$ (that is, assume $\psi(y_1,y_2)=\psi(y_1)$). 
		Define the sequence
		\begin{equation*}
		v_\varepsilon(x,y)=\varepsilon^\alpha\phi(x)\psi\left(\frac{x}{\varepsilon^\alpha}\right).
		\end{equation*}
		Then, we have that
		\begin{eqnarray*}
		&&T_\varepsilon v_\varepsilon\rightarrow 0\mbox{ strongly in }L^p\left((0,1)\times Y^*\right),\\
		&&T_\varepsilon(\partial_x v_\varepsilon)\rightarrow \phi \, \partial_{y_1}\psi \mbox{ strongly in }L^p\left((0,1)\times Y^*\right),\\
		&&T_\varepsilon(\partial_y v_\varepsilon)=0.
		\end{eqnarray*}
		
		Next, taking $v_\varepsilon$ as a test function in \eqref{eq1147}, we have, as $\varepsilon\rightarrow 0$, that
		\begin{eqnarray*}
		\int_{(0,1)\times Y^*}a_0 \cdot (1,0) \, \phi(x) \, \partial_{y_1}\psi \, dxdy_1dy_2=0.
		\end{eqnarray*}
		Hence, by the density of the tensor product $C_0^\infty(0,1)\otimes W^{1,p}_\#(Y^*)$ in $L^p((0,1);W^{1,p}_\#(Y^*))$, we obtain 
		\begin{eqnarray}\label{preauxiliar01}
			\int_{(0,1)\times Y^*}a_0 \cdot (1,0) \, \partial_{y_1}\psi \, dxdy_1dy_2=0, \quad \forall\psi\in L^p((0,1);W^{1,p}_\#(Y^*)),
		\end{eqnarray}
		for all $\psi\in L^p((0,1);W^{1,p}_\#(Y^*))$ such that $\partial_{y_2}\psi=0$. 
%		We can rewrite the above equation as
%		\begin{eqnarray}
%		\int_{(0,1)\times (0,L)}a_0\cdot(1,0)g(y_1)\partial_{y_1}\psi dxdy_1=0,\forall\psi\in L^p((0,1);W^{1,p}_\#(Y^*)).
%		\end{eqnarray}
		
		Now, we identify $a_0$. We argue as in the previous section. We take the limit functions $u\in W^{1,p}(0,1)$ and $u_1\in L^p((0,1);W^{1,p}_\#(Y^*))$, we extend $\partial_{y_1} u_1$ in $y_1$-direction and define the sequence
		\begin{equation} \label{W1186}
		W_\varepsilon(x)=\left(\partial_x u(x)+\partial_{y_1} u_1\left(x,\frac{x}{\varepsilon^\alpha}\right),0\right).
		\end{equation}
		Notice that by Proposition \ref{convergenceplaplacetype} we have
		\begin{eqnarray*}
			&&T_\varepsilon W_\varepsilon\rightarrow \left(\partial_x u+\partial_{y_1} u_1,0\right)\mbox{ in }L^p\left((0,1)\times Y^*\right)^2,\\
			&&T_\varepsilon\left(|W_\varepsilon|^{p-2}W_\varepsilon\right)\rightarrow \left|\partial_x u +\partial_{y_1} u_1\right|^{p-2}\left[\left(\partial_x u+\partial_{y_1} u_1,0\right)\right]\mbox{ in }L^{p'}\left((0,1)\times Y^*\right)^2,
		\end{eqnarray*}
		and then, $T_\varepsilon (\nabla u_\varepsilon - W_\varepsilon) \rightharpoonup 0$ weakly in $L^p((0,1) \times Y^*)$.
		
		Analogously to the resonant case, we have as in \eqref{694} and \eqref{eq829} that
		\begin{eqnarray*}
		&&\nonumber\int_{(0,1)\times Y^*}T_\varepsilon\left[|\nabla u_\varepsilon|^{p-2}\nabla u_\varepsilon-|W_\varepsilon|^{p-2}W_\varepsilon\right]T_\varepsilon\left(\nabla u_\varepsilon-W_\varepsilon\right)dxdy_1dy_2\rightarrow 0,\\
		&&\int_{(0,1)\times Y^*}T_\varepsilon\left[|\nabla u_\varepsilon|^{p-2}\nabla u_\varepsilon-|W_\varepsilon|^{p-2}W_\varepsilon\right]\varphi dxdy_1dy_2\rightarrow 0,
		\end{eqnarray*}
		for any $\varphi\in C_0^\infty\left((0,1)\times Y^*\right)^2$. Therefore, 
		\begin{equation*}
		T_\varepsilon\left(|\nabla u_\varepsilon|^{p-2}\nabla u_\varepsilon\right)\rightharpoonup \left|\partial_x u+\partial_{y_1} u_1\right|^{p-2}\left(\partial_x u +\partial_{y_1} u_1,0\right), \quad \textrm{ weakly in } L^{p'}\left((0,1)\times Y^*\right)^2
		\end{equation*}
		and then, 
		\begin{equation}\label{456}
		a_0=\left|\partial_x u+\partial_{y_1} u_1\right|^{p-2}\left(\partial_x u+\partial_{y_1} u_1,0\right)\mbox{ a.e }(0,1)\times Y^*.
		\end{equation}
		
		Thus, due to \eqref{prehomogenizedalfa<1}, we get 
		\begin{eqnarray}\label{prehomogenized1}
		&&\int_{(0,1)\times Y^*}\left(\hat{f}-|u|^{p-2}u\right)\varphi \, dxdy_1dy_2\\
		&&\qquad \qquad - \nonumber\int_{(0,1)\times Y^*}\left|\partial_xu +\partial_{y_1} u_1\right|^{p-2}\left(\partial_xu+\partial_{y_1} u_1\right)\partial_x \varphi \, dxdy_1dy_2= 0,
		\end{eqnarray}
		for all $\varphi \in W^{1,p}(0,1)$. 
		
		Since $u \in W^{1,p}(0,1)$, $\partial_{y_2}\psi=0$ and $\partial_{y_2} u_1 = 0$, we can take $\psi\in L^p((0,1);W^{1,p}_\#(0,L))$, and then, using \eqref{456}, we can rewrite \eqref{preauxiliar01} as
		\begin{eqnarray*} 
		\int_{(0,1)\times (0,L)}\left|\partial_xu+\partial_{y_1} u_1\right|^{p-2}\left(\partial_xu+\partial_{y_1} u_1\right)\, \partial_{y_1}\psi \, g(y_1) \,dxdy_1=0,
		\end{eqnarray*}
		for all $\psi\in L^p((0,1);W^{1,p}_\#(0,L))$. 
		
		Hence, treating $x$ as a parameter in the above equation we have that there exists a function $T$ depending only on $x$ such that
		\begin{equation*}
		T(x)=\left|\partial_xu+\partial_{y_1} u_1\right|^{p-2}\left(\partial_xu+\partial_{y_1} u_1\right) \, g(y_1)\mbox{ a.e }(0,L).
		\end{equation*}
		By Corollary \ref{corolarioplaplaciano}, the composition of $a_{p}$ and $a_{{p'}}$  gives the identity. Thus, we can rewrite the above equality as
		\begin{equation}
		\label{weaku1}
		\partial_xu+\partial_{y_1} u_1=\left|\dfrac{T}{g}\right|^{p'-2}\dfrac{T}{g}.
		\end{equation}
		Using the periodicity of $u_1$, we get
		\begin{equation*}
		\begin{gathered}
		0=\int_0^L\partial_{y_1} u_1 dy_1=\int_0^L\left|\dfrac{T}{g}\right|^{p'-2}\dfrac{T}{g} -\partial_xu\, dy_1,
		\end{gathered}
		\end{equation*}
		
		which means that
		$$
		|T|^{p'-2}T=\partial_xu/\left\langle \frac{1}{|g|^{p'-2}g}\right\rangle_{(0,L)}
		$$
		Putting this together with \eqref{weaku1}, one can get
		\begin{equation*}
		\partial_{y_1}u_1=\partial_xu \dfrac{1}{|g|^{p'-2}g\left\langle 1/|g|^{p'-2}g\right\rangle_{(0,L)}}-\partial_xu.
		\end{equation*}
		Thus, using the above equality and the fact that $g>0$, we can rewrite \eqref{prehomogenized1} as 
		\begin{equation*}
		\int_0^1\dfrac{1}{\langle g\rangle_{(0,L)}\left\langle 1/g^{p'-1}\right\rangle_{(0,L)}^{p-1}}|\partial_x u|^{p-2}\partial_x u\partial_x \varphi+|u|^{p-2} u\varphi dx=\int_0^1\bar{f} \varphi dx,
		\end{equation*} 
		for all $\varphi\in W^{1,p}(0,1)$, where
		$$
		\bar{f}=\dfrac{1}{|Y^*|}\int_{Y^*}\hat{f}dy_1dy_2.
		$$
	\end{proof}

	As in Corollary \ref{cor1045}, we get the following corrector result to \eqref{problem01} for $\alpha \in (0,1)$:

	\begin{corollary}\label{correctoralfa<1}
		The function $W_\varepsilon$ defined in \eqref{W1186} satisfies
		$$
		\left|\left|T_\varepsilon \nabla u_\varepsilon-T_\varepsilon W_\varepsilon\right|\right|_{L^p\left((0,1)\times Y^*\right)}\rightarrow 0.
		$$
	\end{corollary}
	\begin{remark}
		From Corollary \ref{correctoralfa<1} and \cite[Proposition 2.17]{AM2}, we have that 
		$$
		|||\nabla u_\varepsilon-W_\varepsilon  |||_{L^p(R^\varepsilon)}\rightarrow 0.
		$$
	\end{remark}

	\section{The strong oscillation case: $\alpha>1$.} \label{seal>1}
	
	In this section we analyze the behavior of the solutions of \eqref{problem01} as the upper boundary of the thin domains presents a very high oscillatory boundary. Then, the thin domain is defined as follows
	$$
	R^\varepsilon=\left\{(x,y)\in \mathbb{R}^2: 0<x<1, 0<y<\varepsilon g\left(\dfrac{x}{\varepsilon^\alpha}\right)\right\}.
	$$
	where $\alpha>1$ and $g$ satisfies the hypothesis ($\mathbf{H_g}$).
	
	We would like to point out that even though we use the unfolding operator to get the homogenized limit problem, the approach is different to the two previous cases. The roughness is so strong that we can not obtain a compactness theorem as in the previous results Theorems \ref{thmressonant} and \ref{thmalpha<1}. To overcome this difficulty we will divide the thin domain in two thin parts:
	\begin{eqnarray*}
	&&R_+^\varepsilon=\left\lbrace(x,y)\in \mathbb{R}^2:x\in(0,1),\varepsilon g_0<y<\varepsilon g(x/\varepsilon^\alpha)\right\rbrace,\\
	&&R_-^\varepsilon=\left\lbrace(x,y)\in \mathbb{R}^2:x\in(0,1),0<y<g_0\right\rbrace.
	\end{eqnarray*}
	Notice that 
	\begin{equation*}
	R^\varepsilon=\mbox{ Int }\left(\overline{R_+^\varepsilon}\bigcup \overline{R_-^\varepsilon}\right).
	\end{equation*}
	
	Moreover, let us introduce the following sets independent of $\varepsilon$
	\begin{eqnarray*}
	&&Y^*=\left\lbrace(y_1,y_2)\in\mathbb{R}^2:0<y_1<L,0<y_2<g(y_1)\right\rbrace,\\
	&&Y^*_+=\left\lbrace(y_1,y_2)\in\mathbb{R}^2:0<y_1<L,g_0<y_2<g(y_1)\right\rbrace,\\
	&&R_-=\left\lbrace(x,y)\in\mathbb{R}^2:x\in(0,1),0<y<g_0\right\rbrace,\\
	&&R_+=\left\lbrace(x,y)\in\mathbb{R}^2:x\in(0,1),g_0<y<g_1\right\rbrace.
	\end{eqnarray*}
	
	\begin{remark} Notice that the reference cell for the unfolding operator restricted to the oscillating part, $Y^*_+$, may be disconnected since $g_0 = \displaystyle\min_{x\in\mathbb{R}}{g(x)}.$
	\end{remark}
	
	We first introduce an operator which allows us to rescale $R^\varepsilon_-$ in order to work over a fixed domain.
	
	\begin{definition}
		Let $\varphi\in L^p(R_-^\varepsilon)$. The rescaling operator $\Pi_\varepsilon:L^p(R_-^\varepsilon)\rightarrow L^p(R_-)$ is defined by
		\begin{equation*}
		\Pi_\varepsilon(\varphi)(x,y)=\varphi(x,\varepsilon y), \forall(x,y)\in R_-.
		\end{equation*}
	\end{definition}
	\begin{proposition}\label{rescalingproperties}
		The rescaling operator satisfies the following properties:
		\begin{enumerate}
			\item Let $\varphi \in L^1(R_-^\varepsilon)$. Then
			\begin{equation*}
			\int_{R_-}\Pi_\varepsilon(\varphi)(x,y)dxdy=\frac{1}{\varepsilon}\int_{R^\varepsilon_-}\varphi(x,y)dxdy.
			\end{equation*}
			\item $\Pi_\varepsilon$ is linear and continuous from $L^p(R_-^\varepsilon)$ to $L^p(R_-)$, $1\leq p\leq \infty$. In addition, the following relationship existis between their norms:
			\begin{eqnarray*}
				\left|\left|\Pi_\varepsilon(\cdot)\right|\right|_{L^p(R_-)}=\left|\left|\left|\cdot\right|\right|\right|_{L^p(R_-^\varepsilon)},\\
				\left|\left|\Pi_\varepsilon(\cdot)\right|\right|_{L^\infty(R_-)}=\left|\left|\left|\cdot\right|\right|\right|_{L^\infty(R_-^\varepsilon)}.
			\end{eqnarray*}
			\item For $\varphi\in W^{1,p}(R_-^\varepsilon)$, $1\leq p\leq \infty$, we have 
			\begin{equation*}
			\partial_x\Pi_\varepsilon\varphi=\Pi_\varepsilon\partial_x\varphi, \quad \partial_y\Pi_\varepsilon\varphi=\varepsilon\Pi_\varepsilon\partial_y\varphi.
			\end{equation*}
			\item Let $\phi\in L^p(0,1)$, $1\leq p\leq \infty$. Then, considering $\phi$ as a function defined in $R_-^\varepsilon$, we have $\Pi_\varepsilon\phi=\phi$.
		\end{enumerate}
	\end{proposition}
	\begin{proof}
		The proof follows directly from the definition.
	\end{proof}

	Our homogenization result is the following one.
	\begin{theorem}\label{strongresult}
		Let $u_\varepsilon$ be the solution of the problem \eqref{problem01} with $f^\varepsilon\in L^{p'}(R^\varepsilon)$ and $\left|\left|f^\varepsilon\right|\right|_{L^{p'}(R^\varepsilon)}\leq c$, for some $c>0$ independent of $\varepsilon$. Suppose the function
		$$
		\hat{f}^\varepsilon(x)=\frac{1}{\varepsilon}\int_0^{\varepsilon g\left(\frac{x}{\varepsilon^\alpha}\right)}f(x,y)dy
		$$
		satisfies
		$$
		\hat{f}^\varepsilon\rightharpoonup\hat{f}\mbox{ weakly in }L^{p'}(0,1).
		$$
		
		Then, there exists a unique $u\in W^{1,p}(0,1)$ such that
		\begin{equation*}
		\begin{gathered}
			T_\varepsilon u_\varepsilon\rightharpoonup u \mbox{ weakly in }L^p((0,1);W^{1,p}(Y^*)),\\
	\textrm{ and } \quad 		\left|\left|\left|u_\varepsilon-u\right|\right|\right|_{L^p(R^\varepsilon)}\rightarrow 0,
			\end{gathered}
		\end{equation*}
		as $\varepsilon\rightarrow 0$. 
		
		Moreover, $u$ is the unique solution of the one dimensional $p$-Laplacian problem 
		\begin{equation*}
		\int_{0}^{1}\left(\frac{g_{0}}{\langle g \rangle_{(0,L)}}|\partial_{x}u|^{p-2}\partial_{x}u\partial_{x}\varphi +|u|^{p-2}u \varphi\right) dx=\int_{0}^{1}\bar{f}\varphi dx
		\end{equation*}  
		\begin{equation}  \label{falpha>1}
		\bar f = \frac{\hat f}{\langle g \rangle_{(0,L)}}.
		\end{equation}  
	\end{theorem}
	\begin{proof}
		Throughout this proof, we denote by $T_\varepsilon$ the unfolding operator associated to the cell $Y^*$, $T_\varepsilon: L^p(R^\varepsilon)\rightarrow L^p((0,1)\times Y^*)$ and by $T_\varepsilon^+$ the unfolding operator associated to the cell $Y^*_+$, $T_\varepsilon^+: L^p(R^\varepsilon_+)\rightarrow L^p((0,1)\times Y^*_+)$.
		
		By Proposition \ref{uniformlimtation}, we have uniform bound of solutions. Therefore, from Proposition \ref{propositionconvergence}, there exists $u\in W^{1,p}(0,1)$ and $a_0\in L^{p'}\left((0,1)\times Y^*\right)^2$ such that, up to subsequences,
		\begin{eqnarray*}
		&&\lim_{\varepsilon\rightarrow 0}\left|\left|\left|u_\varepsilon-u\right|\right|\right|_{L^{p}(R^\varepsilon)}=0\\
		&&T_\varepsilon u_\varepsilon\rightharpoonup u \mbox{, weakly in } L^p((0,1);W^{1,p}(Y^*))\\
		&&T_\varepsilon u_\varepsilon\rightarrow u\mbox{, strongly in } L^p((0,1)\times Y^*),\\
		&&T_\varepsilon\left(|\nabla u_\varepsilon|^{p-2}\nabla u_\varepsilon\right)\rightharpoonup a_0 \mbox{ weakly in }L^{p'}\left((0,1)\times Y^*\right)^2.
		\end{eqnarray*}
		
		In order to simplify the notations, we denote the following restrictions by
		\begin{eqnarray*}
			u_\varepsilon^+=u_\varepsilon|_{R^\varepsilon_+} \quad   \textrm{ and } \quad 			
			u_\varepsilon^-=u_\varepsilon|_{R^\varepsilon_-}.
		\end{eqnarray*}
		
		From Proposition \ref{uniformlimtation}, we have that $T_\varepsilon\left(|\nabla u_\varepsilon^+|^{p-2}\nabla u_\varepsilon^+\right)$ is uniformly bounded. Therefore, there exists $u_1\in L^p\left((0,1)\times Y^*\right)^2$ such that, up to subsequences, 
		\begin{equation} \label{1431}
		T_\varepsilon\left(|\nabla u_\varepsilon^+|^{p-2}\partial_{x}u_\varepsilon^+\right)\rightharpoonup u_1\mbox{ weakly in } L^p\left((0,1)\times Y^*\right).
		\end{equation}
		We show that $u_1(x,y_1,y_2)=0$ a.e. in $(0,1)\times Y^*_+$. To do this, we proceed as in \cite[Theorem 5.3]{AM2} using the suitable test functions there introduced.
		
		Since $g_0=\displaystyle\min_{x\in\mathbb{R}}g(x)$ and $g$ is $L$-periodic, there is, at least, a point $y_0$ such that $g(y_0)=g_0$. Furthermore, 
		\begin{equation}\label{linesegment}
		\left\lbrace(y_0,y_2)\in\mathbb{R}^2:y_2\in(g_0,g_1)\right\rbrace\cap Y^*_+=\emptyset.
		\end{equation}
		
		We analyze two cases: $y_0>0$ and $y_0=0$.
		
		First, suppose $y_0>0$. Then, for any $\phi\in C_0^\infty(0,y_0)$ we define the following function:
		\begin{equation}\label{psi}
		\psi(y_1)=\left\lbrace\begin{array}{ll}
		\displaystyle\int_0^{y_1}\phi(t)dt\mbox{, if }0\leq y_1< y_0,\\
		0\mbox{, if }y_0<y_1<L.
		\end{array}\right. 
		\end{equation}
		
		Notice that $\psi$ can be extended by $L$-periodicity and $\psi\in C^\infty[0,y_0)\cup C^\infty(y_0,L)$. Then, we consider the test function
		\begin{equation*}
		\varphi^\varepsilon(x,y)=\varepsilon^\alpha \tilde{\varphi}\left(x,\frac{y}{\varepsilon}\right)\psi\left(\left\{\frac{x}{\varepsilon^\alpha}\right\}\right),\quad(x,y)\in R^\varepsilon,
		\end{equation*}
		where $\varphi\in C_0^\infty(R_+)$, $\psi$ is defined in \eqref{psi} and $\tilde{\cdot}$ is the extension by zero. Using \eqref{linesegment} and the definition of \eqref{psi}, we have that $\varphi^\varepsilon$ is continuous in $R^\varepsilon$ for each $\varepsilon$.
		
		Now, we apply the unfolding operator to the restricition of $\varphi^\varepsilon$ to the thin domain $R^\varepsilon_+$:
		\begin{equation*}
		T_\varepsilon^+(\varphi^\varepsilon)(x,y_1,y_2)=\left\lbrace\begin{array}{ll}
		\varepsilon^\alpha\varphi\left(\varepsilon^\alpha\left[\frac{x}{\varepsilon^\alpha}\right]L+\varepsilon^\alpha y_1,y_2\right)\psi\left(y_1\right)\mbox{ for }(x,y_1,y_2)\in I_\varepsilon\times Y^*_+,\\
		0\mbox{ for }(x,y_1,y_2)\in \Lambda_\varepsilon\times Y^*_+.
		\end{array}\right.
		\end{equation*}
		By Proposition \ref{unfoldingproperties}, we have 
		\begin{eqnarray*}
			T_\varepsilon^+(\partial_x \varphi^\varepsilon)=\frac{1}{\varepsilon^\alpha}\partial_{y_1}T_\varepsilon^+\varphi^\varepsilon=\varepsilon^\alpha T_\varepsilon^+(\partial_x \varphi)\psi(y_1)+\psi'(y_1)T_\varepsilon^+(\varphi(\cdot,\cdot/\varepsilon)),\\
\textrm{ and } \quad 			T_\varepsilon^+(\partial_y\varphi^\varepsilon)=\frac{1}{\varepsilon}\partial_{y_2}T_\varepsilon^+\varphi^\varepsilon=\varepsilon^{\alpha-1}T_\varepsilon^+(\partial_y\varphi(\cdot,\cdot/\varepsilon))\psi(y_1).
		\end{eqnarray*}
		
		Since $\alpha>1$, we obtain from Proposition \ref{convergenceplaplacetype} that 
		\begin{equation} \label{1472}
		\begin{gathered}
			T_\varepsilon^+\varphi^\varepsilon\rightarrow 0\mbox{ strongly in }L^p\left((0,1)\times Y^*_+\right),\\
			T_\varepsilon^+(\partial_x \varphi^\varepsilon)\rightarrow\psi'(y_1)\varphi(x,y_2)\mbox{ strongly in }L^p\left((0,1)\times Y^*_+\right),\\
			T_\varepsilon^+(\partial_y\varphi^\varepsilon)\rightarrow 0\mbox{ strongly in }L^p\left((0,1)\times Y^*_+\right).
			\end{gathered}
		\end{equation}
		
		We use the fact that $\varphi^\varepsilon$ annihilates in $R_-^\varepsilon$ to take it as a test function in \eqref{variationalproblem01} to obtain
		\begin{equation*}
		\int_{R^\varepsilon_+}|\nabla u_\varepsilon|^{p-2}\nabla u_\varepsilon \nabla \varphi^\varepsilon +|u_\varepsilon|^{p-2}u_\varepsilon\varphi^\varepsilon dxdy=\int_{R^\varepsilon_+}f^\varepsilon\varphi^\varepsilon dxdy.
		\end{equation*}
		
		Thus, from \eqref{1431} and \eqref{1472}, we get 
		\begin{equation*}
		\int_{(0,1)\times Y^*_+}u_1\psi'(y_1)\varphi(x,y_2)dxdy_1dy_2=0.
		\end{equation*}
		That is,
		\begin{equation*}
		\int_{(0,1)\times(g_0,g_1)}\varphi(x,y_2)\left(\int_0^L \tilde{u}_1\psi'dy_1\right)dy_2dx=0, \quad \forall \varphi \in C_0^\infty(R_+).
		\end{equation*}
		Thus,
		\begin{equation*}
		\int_0^L \tilde{u}_1(x,y_1,y)\psi'dy_1=0 \quad \mbox{ a.e. }(x,y)\in R_+,
		\end{equation*}
		and then, from \eqref{psi}, we have 
		\begin{equation*}
		\int_0^{y_0} \tilde{u}_1(x,y_1,y) \, \phi(y_1) \, dy_1=0 \quad \mbox{ a.e. }(x,y)\in R_+ \textrm{ and } \forall\phi\in C_0^\infty(0,y_0).
		\end{equation*}
		Hence, we have
		\begin{equation*}
		\tilde{u}_1(x,y_1,y_2)=0\mbox{ a.e. }(x,y_1,y_2)\in (0,1)\times (0,y_0)\times (g_0,g_1).
		\end{equation*}
		
		Now, repeating the arguments to $\psi$ defined as 
		\begin{equation*}\label{psi1}
		\psi(y_1)=\left\lbrace\begin{array}{ll}
		0\mbox{, if }0\leq x< y_0,\\
		\displaystyle\int_{y_0}^{y_1}\phi(t)dt-\int_{y_0}^L\phi(t)dt\mbox{, if }y_0<x<L,
		\end{array}\right. 
		\end{equation*}
		with $\phi\in C_0^\infty(y_0,L)$, we get
		\begin{equation*}
		\tilde{u}_1(x,y_1,y_2)=0\mbox{ a.e. }(x,y_1,y_2)\in (0,1)\times (y_0,L)\times (g_0,g_1).
		\end{equation*}
		Therefore,
		\begin{equation*}
		u_1(x,y_1,y_2)=0 \mbox{ a.e. }(x,y_1,y_2)\in (0,1)\times Y^*_+.
		\end{equation*}
		
		Finally, for $y_0=0$, set
		\begin{eqnarray*}
		\psi(y_1)=\int_0^{y_1}\phi(t)dt,\mbox{ with }0<y_1<L,
		\end{eqnarray*}
		for any $\phi\in C_0^\infty(0,L)$. In this case, we have
		$$
		\left\{(0,y_2):y_2\in (g_0,g_1)\right\}\cap Y^*_+=\emptyset\mbox{, and }\left\{(L,y_2):y_2\in (g_0,g_1)\right\}\cap Y^*_+=\emptyset.
		$$
		Then, the sequence
		$$
		\varphi^\varepsilon(x,y)=\varepsilon^\alpha\tilde{\varphi}\left(x,\frac{y}{\varepsilon}\right)\psi\left(\left\{\frac{x}{\varepsilon^\alpha}\right\}\right),\quad (x,y)\in R^\varepsilon,
		$$
		is well defined since $\varphi \in C^\infty_0(R_+)$.
		
		Thus, using the same arguments as before, we get that $u_1=0$ a.e. $(x,y_1,y_2) \in (0,1)\times Y^*_+$.
		Therefore, 
		\begin{equation}\label{auxiliaruepsilon}
		T_\varepsilon(|\nabla u^+_\varepsilon|^{p-2}\partial_x u^+_\varepsilon)\rightharpoonup 0 \quad \mbox{ weakly in }L^p\left((0,1)\times Y^*_+\right),
		\end{equation}
		and then, 
		$$
		a_0\cdot(1,0)=u_1=0 \quad \mbox{ a.e. }(x,y_1,y_2)\in (0,1)\times Y^*_+.
		$$
		
		For the function $u_\varepsilon^-$, due to Propositions \ref{uniformlimtation} and \ref{rescalingproperties}, there exists $u^-\in W^{1,p}(0,1)$ such that, up to subsequences,
		\begin{eqnarray*}
		&&\Pi_\varepsilon u_\varepsilon^-\rightharpoonup u^- \mbox{ weakly in }W^{1,p}(R_-),\\
		&&\Pi_\varepsilon \partial_x u_\varepsilon^-\rightharpoonup\partial_x u^- \mbox{ weakly in } L^p(R_-),\\
		&&\Pi_\varepsilon \partial_y u_\varepsilon^-\rightharpoonup 0 \mbox{ weakly in } L^p(R_-).
		\end{eqnarray*}
		
		Furthermore, from items (2) and (4) of Proposition \ref{rescalingproperties}, we have 
		\begin{equation*}
		\left|\left|\Pi_\varepsilon u_\varepsilon^--u\right|\right|_{L^p(R_-)}=\left|\left|\left|u_\varepsilon^--u\right|\right|\right|_{L^p(R_-^\varepsilon)}\leq \left|\left|\left|u_\varepsilon-u\right|\right|\right|_{L^p(R^\varepsilon)}.
		\end{equation*}
		Thus,  
		\begin{equation*}\label{4000}
		\Pi_\varepsilon u_\varepsilon^-\rightarrow u\mbox{ strongly in }L^p(R_-),
		\end{equation*}
		which means that 
		$$
		u(x)=u_-(x)\mbox{ a.e. in }(0,1).
		$$
		
		Next, we pass to the limit in $\Pi_\varepsilon\left(|\nabla u_\varepsilon^-|^{p-2}\nabla u_\varepsilon^-\right)$. We show that 
		$$
		\Pi_\varepsilon\left(|\nabla u_\varepsilon^-|^{p-2}\nabla u_\varepsilon^-\right)\rightharpoonup |u'|^{p-2}(u',0) \mbox{ in }L^p(R_-)^2.
		$$
		In order to do that, suppose that $a_1$ is the limit of $\Pi_\varepsilon\left(|\nabla u_\varepsilon^-|^{p-2}\nabla u_\varepsilon^-\right)$, up to subsequences. Again, we use monotonicity arguments. 
		Let $w\in W^{1,p}(0,1)$. 
		
		Using $T_\varepsilon$, $T_\varepsilon^+$ and the rescaling operator $\Pi_\varepsilon$ in the variational formulation \eqref{variationalproblem01}, we obtain 
		\begin{eqnarray}\label{variationalunfalfabig}
		&&\nonumber\frac{1}{L}\int_{(0,1)\times Y^*_+}T_\varepsilon^+\left(|\nabla u_\varepsilon^+|^{p-2}\nabla u_\varepsilon^+\right)T_\varepsilon^+\nabla \varphi dxdy_1dy_2\\\nonumber
		&&+\frac{1}{\varepsilon}\int_{R_{1+}^\varepsilon}|\nabla u_\varepsilon^+|^{p-2}\nabla u_\varepsilon^+\nabla \varphi dxdy\\\nonumber
		&&+\int_{R_-}\Pi_\varepsilon\left(|\nabla u_\varepsilon^-|^{p-2}\nabla u_\varepsilon^-\right)\Pi_\varepsilon\nabla \varphi dxdy\\\nonumber
		&&+\frac{1}{L}\int_{(0,1)\times Y^*}T_\varepsilon(|u_\varepsilon|^{p-2}u_\varepsilon)T_\varepsilon\varphi dxdy_1dy_2\\\nonumber
		&&+\frac{1}{\varepsilon}\int_{R_1^\varepsilon}|u_\varepsilon|^{p-2}u_\varepsilon\varphi dxdy\\
		&&=\frac{1}{\varepsilon}\int_{R^\varepsilon}f^\varepsilon \varphi dxdy.
		\end{eqnarray}
		Hence, we can pass to the limit taking test functions $\varphi\in W^{1,p}(0,1)$ to obtain 
		\begin{equation*}\label{prehomogenizedalfabig}
		\int_{R_-}a_1(\varphi',0)dxdy+\dfrac{1}{L}\int_{(0,1)\times Y^*}|u|^{p-2}u\varphi dxdy_1dy_2=\int_0^1 \hat{f}\varphi dx.
		\end{equation*}
		
		Now, notice that
		\begin{equation*}\label{ineqalfabig1}
		0\leq\dfrac{1}{L}\int_{(0,1)\times Y^*_+}|T_\varepsilon^+\nabla u_\varepsilon^+|^p dxdy_1dy_2,
		\end{equation*}
		also, due to \eqref{auxiliaruepsilon}, we have  
		\begin{eqnarray*}\label{ineqalfabig2}
		0= \lim_{\varepsilon\rightarrow 0}\left(\dfrac{1}{L}\int_{(0,1)\times Y^*_+}|T_\varepsilon^+\nabla u_\varepsilon^+|^{p-2}T_\varepsilon^+\nabla u_\varepsilon^+(T_\varepsilon^+u',0) dxdy_1dy_2\right)
		\end{eqnarray*}
		since $T_\varepsilon^+u' \to u'$ strongly in $L^p((0,1) \times Y^*)$.
		
		Using the monotonicity of $|\cdot|^{p-2}\cdot$, we have 
		\begin{eqnarray}
			0&\leq& \int_{R_-}\left(\Pi_\varepsilon\left(|\nabla u_\varepsilon^-|^{p-2}\nabla u_\varepsilon^-\right)-\Pi_\varepsilon(|u'|^{p-2}(u',0))\right)\left(\Pi_\varepsilon \nabla u_\varepsilon^--\Pi_\varepsilon(u',0)\right) dxdy \nonumber \\
			&\leq&
			\int_{R_-}\left(\Pi_\varepsilon\left(|\nabla u_\varepsilon^-|^{p-2}\nabla u_\varepsilon^-\right) -\Pi_\varepsilon(|u'|^{p-2}(u',0))\right)\left(\Pi_\varepsilon \nabla u_\varepsilon^--\Pi_\varepsilon(u',0)\right) dxdy \nonumber \\
			&&+ \dfrac{1}{L}\int_{(0,1)\times Y^*_+}|T_\varepsilon^+\nabla u_\varepsilon^+|^p dxdy_1dy_2-\dfrac{1}{L}\int_{(0,1)\times Y^*_+}|T_\varepsilon^+\nabla u_\varepsilon^+|^{p-2}T_\varepsilon^+\nabla u_\varepsilon^+(T_\varepsilon^+u',0) dxdy_1dy_2 \nonumber \\
			&&+ \dfrac{1}{L}\int_{(0,1)\times Y^*_+}|T_\varepsilon^+\nabla u_\varepsilon^+|^{p-2}T_\varepsilon^+\nabla u_\varepsilon^+(T_\varepsilon^+u',0) dxdy_1dy_2  \label{1607}
		\end{eqnarray}
		
		Now, we can pass to the limit in \eqref{1607} using the variational formulation \eqref{variationalunfalfabig}. Taking the test functions $\varphi=u_\varepsilon-u$ in \eqref{variationalunfalfabig}, we have that
		\begin{equation} \label{1612}
		\begin{gathered}
		\dfrac{1}{L}\int_{(0,1)\times Y^*_+}|T_\varepsilon^+\nabla u_\varepsilon^+|^p dxdy_1dy_2-\dfrac{1}{L}\int_{(0,1)\times Y^*_+}|T_\varepsilon^+\nabla u_\varepsilon^+|^{p-2}T_\varepsilon^+\nabla u_\varepsilon^+(T_\varepsilon^+u',0) dxdy_1dy_2 \\
		+ \int_{R_-} \Pi_\varepsilon\left(|\nabla u_\varepsilon^-|^{p-2}\nabla u_\varepsilon^-\right) \left(\Pi_\varepsilon \nabla u_\varepsilon^--\Pi_\varepsilon(u',0)\right) dxdy \to 0
		\end{gathered}
		\end{equation}
		since $T_\varepsilon( u_\varepsilon - u) \to 0$ strongly in $L^p((0,1) \times Y^*)$.
		Hence, as $\Pi_\varepsilon \nabla u_\varepsilon^--\Pi_\varepsilon(u',0) \rightharpoonup 0$ weakly in $L^p(R_-)$, we get from \eqref{1607} and \eqref{1612} that 
		\begin{equation} \label{1621}
		\lim_{\varepsilon\rightarrow 0}\int_{R_-}\left(\Pi_\varepsilon\left(|\nabla u_\varepsilon^-|^{p-2}\nabla u_\varepsilon^-\right)-\Pi_\varepsilon(|u'|^{p-2}(u',0))\right)\left(\Pi_\varepsilon \nabla u_\varepsilon^--\Pi_\varepsilon(u',0)\right) dxdy=0.
		\end{equation}
		
		On the other hand, we have
		\begin{eqnarray*}
			\int_{R_-}|\Pi_\varepsilon\nabla u_\varepsilon^--(u',0)|^pdxdy&\leq& c \int_{R_-}|\Pi_\varepsilon\nabla u_\varepsilon^--\Pi_\varepsilon(u',0)|^pdxdy\\
			&&+c\int_{R_-}|\Pi_\varepsilon (u',0)-(u',0)|^pdxdy.
		\end{eqnarray*}
		If $p>2$, we get from \eqref{1621} that 
		\begin{eqnarray*}
			&&\int_{R_-}|\Pi_\varepsilon\nabla u_\varepsilon^--\Pi_\varepsilon(u',0)|^pdxdy\leq\\ 
			&& \qquad c\int_{R_-}\left(\Pi_\varepsilon\left(|\nabla u_\varepsilon^-|^{p-2}\nabla u_\varepsilon^-\right)-\Pi_\varepsilon(|u'|^{p-2}(u',0))\right)\left(\Pi_\varepsilon \nabla u_\varepsilon^--\Pi_\varepsilon(u',0)\right) dxdy\\
			&&\qquad \rightarrow 0.
		\end{eqnarray*}
		For $1<p\leq 2$, using H\"older's inequality and \eqref{1621}, we obtain
		\begin{eqnarray*}
			&&\int_{R_-}|\Pi_\varepsilon\nabla u_\varepsilon^--\Pi_\varepsilon(u',0)|^pdxdy\\
			&=&\int_{R_-}\dfrac{\left(|\Pi_\varepsilon\nabla u_\varepsilon^-|+|\Pi_\varepsilon(u',0)|\right)^{\frac{(p-2)p}{2}}}{\left(|\Pi_\varepsilon\nabla u_\varepsilon^-|+|\Pi_\varepsilon(u',0)|\right)^{\frac{(p-2)p}{2}}}|\Pi_\varepsilon\nabla u_\varepsilon^--\Pi_\varepsilon(u',0)|^pdxdy\\
			&\leq&\left(\int_{R_-}\left(|\Pi_\varepsilon\nabla u_\varepsilon^-|+|\Pi_\varepsilon(u',0)|\right)^{(p-2)}|\Pi_\varepsilon\nabla u_\varepsilon^--\Pi_\varepsilon(u',0)|^2dxdy\right)^{p/2}\\
			&&\left(\int_{R_-}\left(|\Pi_\varepsilon\nabla u_\varepsilon^-|+|\Pi_\varepsilon(u',0)|\right)^{p}dxdy\right)^{(2-p)/2}
		\end{eqnarray*}
		\begin{eqnarray*}
			&\leq&\left(c\int_{R_-}\left(\Pi_\varepsilon\left(|\nabla u_\varepsilon^-|^{p-2}\nabla u_\varepsilon^-\right)-\Pi_\varepsilon(|u'|^{p-2}(u',0))\right)\left(\Pi_\varepsilon \nabla u_\varepsilon^--\Pi_\varepsilon(u',0)\right) dxdy\right)^{p/2}\\
			&&\left(\int_{R_-}\left(|\Pi_\varepsilon\nabla u_\varepsilon^-|+|\Pi_\varepsilon(u',0)|\right)^{p}dxdy\right)^{(2-p)/2}\\
			&\rightarrow& 0.
		\end{eqnarray*}

		Thus, as $\int_{R_-}|\Pi_\varepsilon (u',0)-(u',0)|^pdxdy \to 0$, it follows that 
		\begin{equation*}
		\Pi_\varepsilon \nabla u_\varepsilon^-\rightarrow (u',0)\mbox{ strongly in }L^p(R_-)\times L^p(R_-),
		\end{equation*}
		and then, 
		\begin{equation*}
		|\Pi_\varepsilon \nabla u_\varepsilon^-|^{p-2}\Pi_\varepsilon \nabla u_\varepsilon^-\rightarrow |u'|^{p-2}(u',0)\mbox{ strongly in }L^{p'}(R_-)\times L^{p'}(R_-).
		\end{equation*}
		
		Finally, we can pass to the limit in \eqref{variationalunfalfabig}, for any test functions $\varphi\in W^{1,p}(0,1)$ to obtain  
		\begin{equation*}
		\int_{R_-}|u'|^{p-2}u'\varphi'dxdy+\dfrac{1}{L}\int_{(0,1)\times Y^*}|u|^{p-2}u\varphi dxdy_1dy_2=\int_0^1\hat{f}\varphi dx,
		\end{equation*}
		which is
		\begin{equation*}
		\int_0^1g_0|u'|^{p-2}u'\varphi'+\dfrac{|Y^*|}{L}|u|^{p-2}u\varphi dx=\int_0^1\hat{f}\varphi dx,
		\end{equation*}
		for all $\varphi\in W^{1,p}(0,1)$, concluding the proof.
	\end{proof}
	\begin{corollary} \label{correcal>1}
		Under hypothesis of Theorem \ref{strongresult}, we have that the solutions of problem \eqref{problem01} satisfy 
		\begin{eqnarray}
			\label{convalfa>11}&&\Pi_\varepsilon u_\varepsilon^-\rightarrow u \mbox{ strongly in } W^{1,p}(R_-)\\
			&&\label{convalfa>12}\Pi_\varepsilon \left(|\nabla u_\varepsilon^-|^{p-2}\partial_xu_\varepsilon^-\right)\rightarrow |u'|^{p-2}u'\mbox{ strongly in } L^{p}(R_-)\\
			&&\label{convalfa>13}\Pi_\varepsilon \left(|\nabla u_\varepsilon^-|^{p-2}\partial_y u_\varepsilon^-\right)\rightarrow 0\mbox{ strongly in } L^{p}(R_-)\\
			&&\label{convalfa>14}T_\varepsilon^+\left(|\nabla u_\varepsilon^+|^{p-2}\partial_x u_\varepsilon^+\right) \rightarrow 0\mbox{ strongly in } L^{p}\left((0,1)\times Y^*_+\right)\\
			&&\label{convalfa>15}T_\varepsilon^+\left(|\nabla u_\varepsilon^+|^{p-2}\partial_y u_\varepsilon^+\right) \rightarrow 0\mbox{ strongly in } L^{p}\left((0,1)\times Y^*_+\right)
		\end{eqnarray}
		
		Moreover, 
		$$
		\left|\left|\left|\nabla u_\varepsilon\right|\right|\right|_{L^p(R^\varepsilon_+)}\rightarrow0.
		$$
	\end{corollary}
	\begin{proof}
		The convergence of \eqref{convalfa>11} - \eqref{convalfa>13} are proved in Theorem \ref{strongresult}.
		For the convergence \eqref{convalfa>14} and \eqref{convalfa>15}, we take as test function $(u_{\varepsilon}-u)$ in \eqref{variationalunfalfabig}
%		\begin{eqnarray*}\label{variationalunfalfabig000}
%		&&\nonumber\frac{1}{L}\int_{(0,1)\times Y^*_+}T_\varepsilon^+\left(|\nabla u_\varepsilon^+|^{p-2}\nabla u_\varepsilon^+\right)T_\varepsilon^+(\nabla u_{\varepsilon}-\nabla u)dxdy_1dy_2\\\nonumber
%		&=&\frac{1}{\varepsilon}\int_{R^\varepsilon}f^\varepsilon ( u_{\varepsilon}-u) dxdy\\
%		&&-\frac{1}{\varepsilon}\int_{R_{1+}^\varepsilon}|\nabla u_\varepsilon^+|^{p-2}\nabla u_\varepsilon^+(\nabla u_{\varepsilon}-\nabla u) dxdy\\\nonumber
%		&&-\int_{R_-}\Pi_\varepsilon\left(|\nabla u_\varepsilon^-|^{p-2}\nabla u_\varepsilon^-\right)\Pi_\varepsilon(\nabla u_{\varepsilon}-\nabla u)dxdy\\\nonumber
%		&&-\frac{1}{L}\int_{(0,1)\times Y^*}T_\varepsilon(|u_\varepsilon|^{p-2}u_\varepsilon)T_\varepsilon( u_{\varepsilon}-u) dxdy_1dy_2\\\nonumber
%		&&-\frac{1}{\varepsilon}\int_{R_1^\varepsilon}|u_\varepsilon|^{p-2}u_\varepsilon( u_{\varepsilon}-u)dxdy
%		\end{eqnarray*}
		to get from \eqref{convalfa>11} - \eqref{convalfa>13} and \eqref{auxiliaruepsilon} that 
		\begin{equation*}
		\int_{(0,1)\times Y^*_+}T_\varepsilon^+\left(|\nabla u_\varepsilon^+|^{p-2}\nabla u_\varepsilon^+\right)T_\varepsilon^+(\nabla u_{\varepsilon}-\nabla u)dxdy_1dy_2\rightarrow 0.
		\end{equation*}

		Thus
		\begin{equation*}
			\int_{(0,1)\times Y^*_+}\left|T_\varepsilon^+\nabla u_\varepsilon^+\right|^{p}\rightarrow 0,
		\end{equation*}
		which proves \eqref{convalfa>14} and \eqref{convalfa>15}.
		
		The last convergence follows from the above convergences and  \cite[Proposition 2.17]{AM2}.
		\end{proof}

	\vspace{0.7 cm}

{\bf Acknowledgements.} 
The first author (JMA) is partially supported by grants MTM2016-75465-P, ICMAT Severo Ochoa project SEV-2015-0554, MINECO, Spain and Grupo de Investigaci\'on CADEDIF, UCM. The second author (JCN) is partially supported by CNPq (Brazil), and the third author (MCP) by CNPq 302960/2014-7 and FAPESP 2017/02630-2 (Brazil).

	%\nocite{*}

\end{document}